\documentclass[12pt, reqno]{amsart}
\usepackage[linkcolor=blue, citecolor = blue, linktocpage = true]{hyperref}
\hypersetup{colorlinks=true}

\usepackage{amsthm,amsfonts,amsmath,amssymb,amscd} 
\DeclareMathSymbol{\shortminus}{\mathbin}{AMSa}{"39}
\usepackage{verbatim}
\usepackage{float}
\usepackage{wrapfig}
\usepackage{mathtools}
\usepackage[color=orange!20, linecolor=orange]{todonotes}

\usepackage{enumitem}

\usepackage{ytableau}

\usepackage{tikz}
\usepackage{tikz-cd} 
\usetikzlibrary{decorations.markings}
\tikzstyle{vertex}=[circle, draw, inner sep=0pt, minimum size=4pt]

\usetikzlibrary{arrows,automata,calc}
\usepackage{tikz, tikz-3dplot, pgfplots}
\usetikzlibrary[positioning,patterns]
\usepackage{caption}
\usepackage{subcaption}
\usepackage{lscape}
\usepackage{rotating}

\newtheorem{theorem}{Theorem}[section]
\newtheorem{proposition}[theorem]{Proposition}
\newtheorem{lemma}[theorem]{Lemma}
\newtheorem{corollary}[theorem]{Corollary}

\theoremstyle{definition}
\newtheorem{definition}[theorem]{Definition}
\newtheorem{example}{Example}[section]

\newtheorem{remark}{Remark}[section]

\newcommand{\sgn}{\mathrm{sgn}}

\newcommand{\dett}{\mathrm{det}}
\newcommand{\lex}{\mathrm{lex}}

\newcommand{\Deltaw}{\widetilde{\Delta}}
\newcommand{\nablaw}{\widetilde{\nabla}}
\let\oldbigwedge\bigwedge
\let\oldbigvee\bigvee
\let\oldwedge\wedge
\let\oldvee\vee
\renewcommand{\bigwedge}{
  \mathop{
    \vcenter{
      \hbox{\scalebox{0.9}{$\displaystyle\oldbigwedge$}}
    }
  }
}
\renewcommand{\bigvee}{
  \mathop{
    \vcenter{
      \hbox{\scalebox{0.9}{$\displaystyle\oldbigvee$}}
    }
  }
}
\renewcommand{\wedge}{
  \mathop{
    \vcenter{
      \hbox{\scalebox{1.1}{$\displaystyle\oldwedge$}}
    }
  }
}
\renewcommand{\vee}{
  \mathop{
    \vcenter{
      \hbox{\scalebox{1.1}{$\displaystyle\oldvee$}}
    }
  }
}
\newcommand{\wee}{\mathrel{\wedge\kern-4.5pt\vee}}
\newcommand{\bigwee}{\mathrel{\bigwedge\kern-6pt\bigvee}}

\usepackage{graphicx}

\makeatletter
\newcommand\bigDiamond{\mathop{\mathpalette\bigDi@mond\relax}}
\newcommand\bigDi@mond[2]{%
  \vcenter{\hbox{\m@th
    \scalebox{\ifx#1\displaystyle 2\else1.2\fi}{$#1\Diamond$}%
  }}%
}
\makeatother

\def\acts{\curvearrowright}

\DeclarePairedDelimiter\floor{\lfloor}{\rfloor}

\hypersetup{
    colorlinks=true, 
    linktoc=section,     
    linkcolor=blue,  
}
\tikzset{
  edge/.style   ={line width=1.1pt},
  ghost/.style  ={line width=1.1pt, dashed},
  vwhite/.style ={circle, draw, fill=white, inner sep=0pt, minimum size=6pt, line width=1.1pt},
  vblack/.style ={circle, draw=black, fill=black, inner sep=0pt, minimum size=6pt, line width=1.1pt},
  tinyedge/.style={line width=1.1pt} 
}

\newcommand{\picturematrix}{
    \begin{tikzpicture}[x=1.85cm,y=1cm,z=1cm,rotate around y=137,rotate around x=0, rotate around z=0]
        \draw[->, very thin, gray] (0,0,0) -- (\value{L},0,0) node[above] {$x$};
        \draw[->, very thin, gray] (0,0,0) -- (0,3.5*\value{l}+0.5,0) node[above] {$y$};
        \draw[->, very thin, gray] (0,0,0) -- (0,0,4*\value{l}+0.2) node[above] {$z$};
        
        \draw[fill=cyan] (0,1.2,0) circle[radius=2pt] node[anchor=south west]{$\lambda^{(2)}_1 = 5$};
        \draw[fill=cyan] (0,2.2,0) circle[radius=2pt] node[anchor=south west]{$\lambda^{(2)}_2 = 3$};
        \draw[fill=cyan] (0,3.2,0) circle[radius=2pt] node[anchor=south west]{$\lambda^{(2)}_3 = 2$};
        \draw[-, very thin, dashed, cyan] (0, 1.2, 0) -- (2, 1.2, 0) -- (2, 1.2, 4) -- (0, 1.2, 4) -- cycle;
        \draw[-, very thin, dashed, cyan] (1, 1.2, 0) -- (1, 1.2, 4);
        
        \draw[fill=gray] (0,0,0.7) circle[radius=2pt] node[anchor=south west]{$\lambda^{(3)}_1 = 4$};
        \draw[fill=gray] (0,0,1.7) circle[radius=2pt] node[anchor=south west]{$\lambda^{(3)}_2 = 4$};
        \draw[fill=gray] (0,0,2.7) circle[radius=2pt] node[anchor=south west]{$\lambda^{(3)}_3 = 2$};
    
        \drawsquare{1*\value{l}}{0}{0}{color=black, very thick};
        \putcell{1*\value{l}}{1*\value{l}}{1*\value{l}}{black}{2};
        \putcell{1*\value{l}}{3*\value{l}}{1*\value{l}}{black}{1};
        \putcell{1*\value{l}}{2*\value{l}}{3*\value{l}}{black}{2};
        \putcell{1*\value{l}}{3*\value{l}}{2*\value{l}}{black}{1};
        
        \drawsquare{2*\value{l}}{0}{0}{color=black, very thick};
        \putcell{2*\value{l}}{1*\value{l}}{2*\value{l}}{black}{3};
        \putcell{2*\value{l}}{2*\value{l}}{1*\value{l}}{black}{1};
        
        \draw (1,4,1) node[above]{$\lambda^{(1)}_1 = 6$};
        \draw (2,4,1) node[above]{$\lambda^{(1)}_2 = 4$};
    \end{tikzpicture}
}

\tikzset{
  cubedot/.style={circle,draw,fill=white,line width=1pt,inner sep=1.6pt},
  cubebdot/.style={circle,draw,fill=black,line width=1pt,inner sep=1.6pt},
  cubeedge/.style={line width=0.8pt},
  cubedash/.style={cubeedge, dashed, line width=1pt, cyan}
}

\newcommand{\DrawCube}[3][]{%
\begin{scope}[#1]
  \coordinate (000) at (0,0,0);
  \coordinate (100) at (1,0,0);
  \coordinate (010) at (0,1,0);
  \coordinate (110) at (1,1,0);
  \coordinate (001) at (0,0,1);
  \coordinate (101) at (1,0,1);
  \coordinate (011) at (0,1,1);
  \coordinate (111) at (1,1,1);

  \draw[cubeedge] (000)--(100);
  \draw[cubeedge] (010)--(110);
  \draw[cubeedge] (001)--(101);
  \draw[cubeedge] (011)--(111);
  \draw[cubeedge] (000)--(010);
  \draw[cubeedge] (100)--(110);
  \draw[cubeedge] (001)--(011);
  \draw[cubeedge] (101)--(111);
  \draw[cubeedge] (000)--(001);
  \draw[cubeedge] (100)--(101);
  \draw[cubeedge] (010)--(011);
  \draw[cubeedge] (110)--(111);

  \foreach \P in {000,100,010,110,001,101,011,111}
    \node[cubedot] at (\P) {};

  \foreach \P in {#2}
    \node[cubebdot] at (\P) {};
\end{scope}
}

\newcommand{\picturecubes}{
    \begin{tikzpicture}[x=2cm,y=1.5cm,z=1cm,rotate around y=90,rotate around x=0, rotate around z=0]
        \DrawCube{000,010,101,111}{000/010,010/111,111/101,101/000}
        \DrawCube[xshift=6cm]{000,011,101,110}{000/101,000/011,000/110,101/011,101/110,011/110}   
    \end{tikzpicture}
}

\title[Highest weight vectors of tensors]
{Highest weight vectors of tensors} 

\author[Alimzhan Amanov \and Damir Yeliussizov]{Alimzhan Amanov \and Damir Yeliussizov}
\address{KBTU, Almaty, Kazakhstan, and \newline Institute of Mathematics, Polish Academy of Sciences, Poland}
\email{\href{mailto:alimzhan.amanov@gmail.com}{alimzhan.amanov@gmail.com}}
\address{KBTU, Almaty, Kazakhstan} 
\email{\href{mailto:yeldamir@gmail.com}{yeldamir@gmail.com}}

\begin{document}
\maketitle

\begin{abstract}
	We study highest weight vectors for symmetric and alternating spaces of tensors, whose dimensions are given by generalized Kronecker coefficients. 
	We describe the algebraic relations for classical constructions of corresponding spanning sets of highest weight vectors. 
	The proof is based on important 
	duality that we discover for these highest weight spaces and vectors. 
	As applications of duality, we also give conceptual interpretations to power expansions of Cayley's first hyperdeterminant and its dual exterior form. 
\end{abstract}

\tableofcontents

\section{Introduction}
    Consider the natural tensor action of the group $G = \mathrm{GL}(n, \mathbb{C})^{\times d}$ on the space of tensors $V = (\mathbb{C}^n)^{\otimes d}$. Our primary objects of study are highest weight vectors (w.r.t. the action of $G$) for the symmetric and alternating algebras, decomposed into highest weight spaces  
\begin{align}\label{eq1}
\mathrm{HWV}_{\pmb{\lambda}}\, \mathrm{Sym}^m(V) = \mathrm{HWV}_{\pmb{\lambda}} \bigvee^m V, 
\qquad \mathrm{HWV}_{\pmb{\lambda}}\, \mathrm{Alt}^m(V) = \mathrm{HWV}_{\pmb{\lambda}} \bigwedge^m V, 
\end{align}
where $\pmb{\lambda}$ is a $d$-tuple of partitions of $m$, denoting the corresponding weights. 
The dimensions of these spaces are given by the (generalized) {\it Kronecker coefficients} $g(\pmb{\lambda})$ and $g(\pmb{\lambda'})$ respectively (for odd $d$), where $\pmb{\lambda'}$ denotes the conjugate weight (for even $d$, the second coefficient has slightly different weights).

\vspace{0.5em}

In this paper, we study a unified explicit construction of spanning sets of highest weight vectors for the spaces \eqref{eq1}, which we denote by $\{\Delta\}$ (polynomials) and $\{\nabla\}$ (forms) indexed by $\mathbb{N}$-hypermatrices (or $d$-dimensional contingency tables). 

\vspace{0.5em}

Such constructions are classical, cf. \cite{processi, bi13, widg, widg1}, 
and highest weight vectors are important in various areas including representation theory, geometry, complexity, and combinatorics. 
The polynomials $\Delta$ presented in \cite{widg, widg1} were significant in algorithms for tensor scaling and the quantum marginal problem. 
We also studied these polynomials (as invariants, for rectangular weights) in \cite{ayfund} in connection with positivity of Kronecker coefficients. 
In \cite{yel}, they were used for an asymptotic version of the Alon--Tarsi conjecture. 
In \cite{iken15}, highest weight polynomials were used to make progress on the Saxl conjecture.
Highest weight vectors of tensors are important in Geometric Complexity Theory (GCT), see e.g. \cite{bi13, bi, bdi}. 
Lower bounds in GCT can be proved using highest weight vectors, yet their computation is NP-hard \cite{bdi}. 
Moreover, highest weight vectors were used 
in \cite{imw17} to show that deciding positivity of Kronecker coefficients is NP-hard. 


In the special case of rectangular weights, highest weight polynomials become invariants for the action of the group $\mathrm{SL}^{\times d}$. Cayley in his pioneering work \cite{cay2} referred to such invariants as {\it hyperdeterminants}, higher dimensional generalizations of the ordinary determinant. These polynomials are important in algebraic geometry, for instance, vanishing properties of the {geometric hyperdeterminant} studied in \cite{gkz} relate to the generalized notion of degeneracy for tensors (on the other hand, the form of such hyperdeterminant is not explicit in general and its degree grows exponentially). Hyperdeterminants are also useful to study entanglement of multipartite states in quantum physics, see e.g. \cite{miyake}. 

\vspace{0.5em}

One of the main problems that we resolve in this paper, refers to description 
of all algebraic relations among these 
highest weight vectors. As we show, to solve this problem it is important to study both symmetric and alternating spaces and dual highest weight vectors. The explicit forms $\nabla$ which we construct, provide a complementary set of highest weight vectors for the alternating spaces. With this unified approach 
we obtain some fundamental properties of these polynomials and forms. 
In particular, we describe all linear relations for these spanning sets of highest weight vectors.  
Note that this problem is related to the construction of bases for the spaces \eqref{eq1}, which is also related to combinatorial interpretation of Kronecker coefficients, a major open problem in algebraic combinatorics.



\vspace{0.5em}

Our main results are summarized as follows: 

{\bf 1. Algebraic relations.} 
We describe the algebraic relations among highest weight vectors $\{\Delta \}$ and $\{ \nabla \}$ spanning the spaces \eqref{eq1}. The proof is based on the next two important results. 

{\bf 2. Duality isomorphism of highest weight spaces.} 
With highest weight vectors $\{\Delta \}$ and $\{ \nabla \}$, we construct explicit isomorphisms for dual (equidimensional) highest weight spaces.  

{\bf 3. Duality of highest weight vectors.} 
We discover a natural yet nontrivial duality for expansions of highest weight vectors $\{\Delta \}$ and $\{ \nabla \}$. 
The proof of this property relies on constructing some invariants of certain double cosets of Young subgroups. 
The duality is crucial for proving both of the above results. 
In special cases, this duality also gives conceptual interpretations to power expansions of Cayley's first hyperdeterminant and its dual exterior Cayley form; these expansions are related to evaluations of fundamental SL-invariants of tensors and higher-dimensional analogues of the Alon--Tarsi conjecture on Latin hypercubes, and have important implications to positivity of Kronecker coefficients. 

\vspace{0.5em}

Let us now discuss our results in more detail. 

\subsection*{Spanning sets of highest weight vectors} 
For a $d$-tuple $\pmb{\lambda}$ of partitions of $m$, we denote by $\mathsf{T}_{\mathbb{N}}(\pmb{\lambda})$ and $\mathsf{T}_{01}(\pmb{\lambda})$ the sets of $d$-dimensional $\mathbb{N}$- and $(0,1)$-hypermatrices with marginals $\pmb{\lambda}$. The \textit{support} of the hypermatrix $P \in \mathsf{T}_\mathbb{N}(\pmb{\lambda})$ is denoted by $\mathrm{supp}(P) = \{P(1) \le \ldots \le P(m) \}$ with elements $P(i) \in \mathbb{N}^d$ ordered lexicographically, assuming $\mathrm{supp}(P)$ is finite. We write $\bigvee e_P := \bigvee_{i=1}^m e_{P(i)}$ for the $m$-th symmetric product,\footnote{To stress duality we use a less common notation $\bigvee e_P$ instead of the usual product $\prod e_P$. } 
$\bigwedge e_P := \bigwedge_{i=1}^m e_{P(i)}$ for the $m$-th exterior product, and $e_{\mathbf{i}} = e_{i_1} \otimes \ldots \otimes e_{i_d}$ for $\mathbf{i} = (i_1,\ldots, i_d) \in \mathbb{N}^d$, where $\{e_i \}$ is the standard basis of $\mathbb{C}^n$. Then 
\begin{align}\label{eq:intro-wv}
    \{\bigvee e_{S} : S \in \mathsf{T}_{\mathbb{N}}(\pmb{\lambda}) \}\qquad\text{and}\qquad\{\bigwedge e_{T} : T \in \mathsf{T}_{01}(\pmb{\lambda}) \}
\end{align}
are the standard bases of the weight spaces $\mathrm{WV}_{\pmb{\lambda}} \bigvee V$ and $\mathrm{WV}_{\pmb{\lambda}} \bigwedge V$, respectively. 


\subsubsection*{$\Delta, \nabla$ vectors}
    In \textsection\,\ref{sec:spanningtensor} we explicitly construct the vectors $\{\Delta_T \}$, $\{\nabla_S \}$ indexed by $\mathbb{N}$-hypermatrices with marginals $\pmb{\lambda'}$ and show that they span the corresponding highest weight spaces  \eqref{eq1}. Their presentations are of the form 
    \begin{align*} 
        \Delta_T = \pm \frac{1}{||T||} \sum_{X } \sgn_T(X) \bigvee e_X, 
       \quad 
        \nabla_S = \pm \frac{1}{||S||} \sum_{Y } \sgn_S(Y) \bigwedge e_Y, 
    \end{align*}
where $||\cdot||$ denotes the size of the stabilizer w.r.t. $S_{|\pmb{\lambda}|}$, the sums run over $d$-tuples of words of weight $\pmb{\lambda}$ and the sign functions are defined in Def.~\ref{def:sign}. 
In \textsection\,\ref{sec:spanningset} and \textsection\,\ref{sec:spanningtensor}, we present a unified construction of these vectors by symmetric and alternating projections from a basis of highest weight tensor spaces. 

\vspace{0.5em}

Then for odd $d$ we have  
    \begin{align*}
        \mathrm{HWV}_{\pmb{\lambda}} \bigvee V 
        = \mathrm{span}\{ 
            \Delta_T \mid T \in \mathsf{T}_{01}(\pmb{\lambda'})
        \}, \quad
        \mathrm{HWV}_{\pmb{\lambda}} \bigwedge V 
        = \mathrm{span}\{ 
            \nabla_S \mid S \in \mathsf{T}_\mathbb{N}(\pmb{\lambda'})
        \},
    \end{align*}
    where $T$ and $S$ have conjugate weights.
While for even $d$ we have\footnote{As we will see, there is always slightly different behavior for odd and even $d$ cases.}  
\begin{align*}
    \mathrm{HWV}_{\pmb{\lambda}} \bigvee V 
    = \mathrm{span}\{ 
        \Delta_S \mid S \in \mathsf{T}_\mathbb{N}(\pmb{\lambda'})
    \},\quad
    \mathrm{HWV}_{\pmb{\lambda}} \bigwedge V 
    = \mathrm{span}\{ 
        \nabla_T \mid T \in \mathsf{T}_{01}(\pmb{\lambda'})
    \}.
\end{align*}
Comparing these spanning sets with bases of weight spaces in \eqref{eq:intro-wv}, we already see certain duality of index sets, which is the cornerstone of our paper.





\vspace{0.5em}

Let us note that the general method of constructing highest weight vectors is rather classical, cf. \cite{processi, bi13, widg, widg1}. The polynomials $\Delta_T$ were presented in equivalent explicit form in \cite{widg1} in connection with algorithms for tensor scaling. 
We also studied the polynomials $\Delta_T$ (for rectangular weights) in \cite{ayfund} in connection with positivity of Kronecker coefficients. The explicit forms $\nabla_S$ constructed here, then provide a complementary set of highest weight vectors for the alternating spaces. 

\subsection{Algebraic relations}
Throughout the text, we work with (ordered) versions of hypermatrices, \textit{$d$-line notations or tables}, see \textsection\,\ref{subsec:convention}.
For a table (or a $d$-tuple of words) $X$ the operation $\partial^{(\ell)}_{i j}$ produces the set of {\it boundary} tables of $X$ by replacing single letter $j$ in $\ell$-th row (or direction) with letter $i$, see \textsection\,\ref{sec:linear} for details. If $X$ has weight $\pmb{\lambda}$ then each table in $\partial_{ij}^{(\ell)}(X)$ has the same weight $\alpha^{(\ell)}_{ij}\pmb{\lambda}$.

\vspace{0.5em}

We then describe the algebraic relations for highest weight vectors $\Delta$ and $\nabla$ as follows.

\begin{theorem}[Description of the algebraic relations for $\Delta$ and $\nabla$, cf. Theorem~\ref{th:relations1}]\label{th:relations-intro}
For each $m$ and $d$-tuple $\pmb{\lambda}$ of partitions of $m$, 
highest weight vectors $\Delta$ and $\nabla$ satisfy the following linear relations for all $\ell \in [d], i < j $: 
\begin{align*}
\sum_{T \in \partial^{(\ell)}_{j i}(X)} \Delta_T = 0, \qquad 
\sum_{S \in \partial^{(\ell)}_{j i}(Y)} \nabla_S = 0, \quad 
\end{align*}
s.t. {\small $X \in \mathsf{T}_{01}(\alpha^{(\ell)}_{i j} \pmb{\lambda'} ), Y \in \mathsf{T}_{\mathbb{N}}(\alpha^{(\ell)}_{i j} \pmb{\lambda'} )$}
for odd $d$, or {\small $X \in \mathsf{T}_{\mathbb{N}}(\alpha^{(\ell)}_{i j} \pmb{\lambda'} ), Y \in \mathsf{T}_{01}(\alpha^{(\ell)}_{i j} \pmb{\lambda'} )$} for even $d$. \\
Moreover, any algebraic relation\footnote{Meaning relations of the form $\sum \alpha_T \Delta_{T^{(1)}} \vee\ldots \vee\Delta_{T^{(k)}} = 0$ or $\sum \beta_S \nabla_{S^{(1)}}\wedge\ldots\wedge\nabla_{S^{(k)}} = 0$.} among the polynomials $\Delta$ or the forms $\nabla$ is a consequence of these relations. 

\end{theorem}

The proof of this theorem is based on duality theorems~\ref{th:intro2} and \ref{th:intro3} presented below. 
Let us note an important subtlety here: Theorem~\ref{th:relations-intro}  establishes defining relations for concrete spanning sets of highest weight vectors via dual spaces; e.g. for odd $d$, the defining relations of $\pmb{\lambda}$-weight $\Delta$-polynomials are derived via the defining identities\footnote{ Each highest weight vector $P$ 
admits relations on its expansion coefficients coming from defining identities $U \cdot P = P$, where $U$ is the corresponding upper unitriangular part of $\mathrm{GL}^{\times d}$. } 
on coefficients of dual $\pmb{\lambda'}$-weight $\nabla$-forms, and vice versa.  

\begin{remark}
For fixed $m$ and $d$-tuple $\pmb{\lambda}$ of partitions of $m$, the presented finite list of relations gives all linear relations for the corresponding highest weight spaces \eqref{eq1}. 
\end{remark}

\begin{remark}
Even though these relations have linear form, many of them can also be viewed as polynomial relations w.r.t. corresponding multiplication, see \textsection\,\ref{sec:linear}. In a sense, such relations can be regarded as higher analogues of straightening relations. 
\end{remark}




Let us also illustrate a small example for this thoerem.

\begin{example}
\label{ex:2x2x2}
    Let $d = 3, n = 2$ so that $V = \mathbb{C}^2 \otimes \mathbb{C}^2 \otimes \mathbb{C}^2$, and consider the group $H = \mathrm{SL}(2, \mathbb{C})^{\times 3}$ acting on $V$.  
    Up to slice symmetries, there are four distinct $(0,1)$-hypermatrices $A,B,C,X \in \mathsf{T}_{01}((22),(22),(22))$ with marginals $(2,2)$, see Figure~\ref{fig:cubes}. We then have $$\mathbb{C}[V]^H_4 = \mathrm{span}\{\Delta_A, \Delta_B, \Delta_C, \Delta_X\}$$ 
    for the space of $H$-invariant polynomials of degree $4$.  
    The relations from Theorem~\ref{th:relations-intro} can be written as follows:
    \begin{align*}
        \Delta_A - \Delta_B \pm \Delta_X = 0,\\
        \Delta_B - \Delta_C \pm \Delta_X = 0,\\
        \Delta_C - \Delta_A \pm \Delta_X = 0,
    \end{align*}
    which implies that $\Delta_X = 0$ and $\Delta_A = \Delta_B = \Delta_C$ is unique invariant. This polynomial is known as Cayley's hypedeterminant, 
    which in fact generates the entire ring of invariants $\mathbb{C}[V]^H$ \cite{bbs12}.
\end{example}

\begin{remark}
    More generally, one can study the ring $\mathbb{C}[V]^{H}$ of $H=\mathrm{SL}(2)^{\times d}$ polynomial invariants of the space $V=(\mathbb{C}^{2})^{\otimes d}$ (also called the space of $d$ qubits) as the module of $(0,1)$-hypermatrices $\bigoplus_k \mathbb{C}\mathsf{T}_{01}(k \times 2,\ldots,k \times 2)$ modulo such simple relations. 
\end{remark}

\begin{figure}
    \begin{center}
        \centering
        \picturecubes
    \end{center}
    \caption{The hypermatrices $A,B,C$ are rotationally equivalent to the hypermatrix on the left, and $X$ represents the hypermatrix on the right. (Here the black dots correspond to $1$'s and the white dots to $0$'s.) }
    \label{fig:cubes}
\end{figure}

\subsection{Duality isomorphism of highest weight spaces}
We prove that the canonical linear maps defined on highest weight vectors $\Delta$ and $\nabla$ give explicit isomorphisms for corresponding dual highest weight spaces. 



\begin{theorem}[Duality isomorphism of highest weight spaces, cf. Theorem~\ref{th:isomorphism}]\label{th:intro2}
Let $\pmb{\lambda}$ be a $d$-tuple of partitions of $m$.

(i) For odd $d$ the  linear maps 
\begin{align*}
\Delta &: \mathrm{HWV}_{\pmb{\lambda'}} \bigwedge V \longrightarrow \mathrm{HWV}_{\pmb{\lambda}} \bigvee V, \quad \Delta : \bigwedge e_T \longmapsto \Delta_T, \quad T \in \mathsf{T}_{01}(\pmb{\lambda'}), \\
\nabla &: \mathrm{HWV}_{\pmb{\lambda'}} \bigvee V \longrightarrow \mathrm{HWV}_{\pmb{\lambda}} \bigwedge V, \quad \nabla : \bigvee e_S \longmapsto \nabla_S, \quad S \in \mathsf{T}_{\mathbb{N}}(\pmb{\lambda'}),
\end{align*}
are isomorphisms of highest weight spaces.

(ii) For even $d$ the linear maps 
\begin{align*}
\Delta &: \mathrm{HWV}_{\pmb{\lambda'}} \bigvee V \longrightarrow \mathrm{HWV}_{\pmb{\lambda}} \bigvee V, \quad \Delta : \bigvee e_S \longmapsto \Delta_S, \quad S \in \mathsf{T}_{\mathbb{N}}(\pmb{\lambda'}), \\
\nabla &: \mathrm{HWV}_{\pmb{\lambda'}} \bigwedge V \longrightarrow \mathrm{HWV}_{\pmb{\lambda}} \bigwedge V, \quad \nabla : \bigwedge e_T \longmapsto \nabla_T, \quad T \in \mathsf{T}_{01}(\pmb{\lambda'}),
\end{align*}
are isomorphisms of highest weight spaces.
\end{theorem}


\subsection{Duality of highest weight vectors}
The key instrumental result used in the proofs of the above results as well as in the applications below, is the following duality for highest weight vectors $\Delta$ and $\nabla$. 

\begin{theorem}[Duality for $\Delta, \nabla$, cf. Theorem~\ref{th:duality2}]\label{th:intro3}
        Let $\pmb{\lambda}$ be a $d$-tuple of partitions of $m$.  
    
    (i) For odd $d$ we have 
    \begin{align*}
        \Delta_{T} &= \pm 
        \sum_{S \in \mathsf{T}_{\mathbb{N}}(\pmb{\lambda})}~
            \langle\nabla_S, \bigwedge e_T\rangle \bigvee e_S, \quad T \in  \mathsf{T}_{01}(\pmb{\lambda'}), 
            \\
        \nabla_{S} &= \pm 
        \sum_{T \in  \mathsf{T}_{01}(\pmb{\lambda})}
            \langle\Delta_T, \bigvee e_S\rangle \bigwedge e_T, \quad S \in \mathsf{T}_{\mathbb{N}}(\pmb{\lambda'}). 
    \end{align*}
    
    (ii) For even $d$  we have 
    \begin{align*}
        \Delta_{S} &=\pm \sum_{S' \in \mathsf{T}_{\mathbb{N}}(\pmb{\lambda})}~
            \langle
                \Delta_{S'}, 
                \bigvee e_{S}\rangle
            \bigvee e_{S'}, \quad S \in \mathsf{T}_{\mathbb{N}}(\pmb{\lambda'}),
            \\
        \nabla_{T} &=\pm \sum_{T' \in  \mathsf{T}_{01}(\pmb{\lambda})}
            \langle
                \nabla_{T'}, 
                \bigwedge e_T\rangle 
            \bigwedge e_{T'}, \quad T \in \mathsf{T}_{01}(\pmb{\lambda'}). 
    \end{align*}

\end{theorem}

Let us note that while the part (i) for odd $d$ shows the duality between $\Delta$ and $\nabla$, the part (ii) for even $d$ rather expresses a kind of self-duality and symmetry under conjugation for both $\Delta$ and $\nabla$. 

This duality theorem is obtained by a detailed study of the underlying coefficients appearing in the standard bases expansions of $\Delta$ and $\nabla$:
$$
    a(T, S) := \langle \Delta_T, \bigvee e_S \rangle, \qquad 
    b(S, T) := \langle \nabla_S, \bigwedge e_T \rangle,
$$
for which we prove that $a(T,S) = \pm\, b(S,T)$ for odd $d$, and that $a(T, S) = \pm a(S, T),$ $b(S, T) = \pm b(T, S)$ for even $d$ (where the signs depend only on $\pmb{\lambda}$). 
These coefficients can be expressed as certain signed sums over intersections of double cosets of Young subgroups, see \textsection\,\ref{sec:duality}. In particular, we find certain invariants of double cosets of Young subgroups, which help to establish the needed connection. 

Furthermore, the coefficients $a(T, S)$ have the following important properties: 
(i) They show which highest weight vectors (invariant polynomials or forms) do not vanish, which is in general a hard problem \cite{bdi}. 
(ii) The rank of the matrix $\{a(T,S)\}_{(T,S) \in \mathsf{T}_{01}(\pmb{\lambda}) \times \mathsf{T}_{\mathbb{N}}(\pmb{\lambda})}$ is the Kronecker coefficient $g(\pmb{\lambda})$.
(iii) As special cases, $a(T, S)$ also refine: the dimensions $f^{\lambda} = \dim [\lambda]$  of irreducible representations $[\lambda]$ of the symmetric group (which have the hook-length formula), the number of $d$-dimensional Latin hypercubes $|\mathsf{L}_d(k)|$ (which is a high-dimensional generalization of Latin squares), and signed sums over Latin hypercubes called the \textit{$d$-dimensional Alon--Tarsi numbers} $\mathrm{AT}_d(k)$ which correspond to evaluations of certain fundamental invariants at unit tensor.


\subsubsection*{Fundamental invariants and Latin hypercubes}
Regarded as polynomials on the space of tensors $V = (\mathbb{C}^{n})^{\otimes d}$, the vectors $\Delta_T$ of  weight $\pmb{\lambda} = (n \times k)^d$ (where $n \times k := (k, \ldots, k) = (k^n)$) are also $\mathrm{SL}(n)^{\times d}$-invariant. Their evaluations at unit tensors $\mathsf{I}_n := \sum_{i = 1}^n e_{i,\ldots, i} = \sum_{i = 1}^n e_{i} \otimes \cdots \otimes e_i$  are particularly important. Combinatorially, the values $\Delta_T(\mathsf{I}_n)$ can be expressed as certain signed sums over {\it partial} Latin hypercubes \cite{ayfund}. Latin hypercubes are higher-dimensional generalizations of Latin squares which are matrices whose every row and column is a permutation, see \textsection\,\ref{sec:latin} for definitions.  

For $n = k^{d-1}$ and odd $d \ge 3$ there is a unique (up to scalar) {\it fundamental invariant} $\Delta_{d,k} \in \mathbb{C}[V]^{\mathrm{SL}(n)^{\times d}}_{nk}$ of degree $k^{d}$, introduced by  B\"urgisser and Ikenmeyer \cite{bi} (for $d = 3$) and studied in \cite{ayfund} (for general $d$); see also \cite{LZX21} for related work. The polynomial $\Delta_{d,k}$ can be defined as $\Delta_T$ for the hypermatrix $T$ with $\mathrm{supp}(T) = [k]^d$.  Its evaluation at the unit tensor $\Delta_{d,k}(\mathsf{I}_n) =: \mathrm{AT}_d(k)$ can be expressed as a signed sum over Latin hypercubes, which we call the {$d$-dimensional Alon--Tarsi number}. The problem whether $\mathrm{AT}_3(k) \ne 0$ for even $k$ was posed by B\"urgisser--Ikenmeyer in \cite{bi}, which is a $3$-dimensional analogue of the celebrated Alon--Tarsi conjecture stating that $\mathrm{AT}_2(k) \ne 0$ or that the number of $k \times k$ Latin squares 
of positive and negative signs are not equal, see \cite{at, hr}. 



\subsection{The Cayley form and first hyperdeterminant} 
For odd $d$ we define the \textit{Cayley form} $\omega = \omega_{d,k} \in \bigwedge^k (\mathbb{C}^k)^{\otimes d}$ as follows:
$$
    \omega := \frac{1}{k!} \sum_{\pi_1,\ldots, \pi_d \in S_k} \sgn(\pi_1)\cdots\sgn(\pi_d) \bigwedge_{i=1}^k e_{\pi_1(i),\ldots,\pi_d(i)}. 
$$
Similarly, for even $d$ we define {\it Cayley's first hyperdeterminant} $\delta = \delta_{d,k} \in \bigvee^k (\mathbb{C}^k)^{\otimes d}$ as follows: 
$$
\delta := \frac{1}{k!} \sum_{\pi_1,\ldots, \pi_d \in S_k} \sgn(\pi_1)\cdots\sgn(\pi_d) \bigvee_{i=1}^k e_{\pi_1(i),\ldots,\pi_d(i)}.
$$
(Note that for even (resp. odd) $d$ the corresponding r.h.s. for $\omega$ (resp. $\delta$) vanish.)


\vspace{0.5em}

It can be shown that $\omega \in \mathrm{HWV}_{(k \times 1)^d}\bigwedge^k (\mathbb{C}^k)^{\otimes d}$ is unique (up scalar) highest weight vector in that space for odd $d$ \cite{ayfund}. Moreover, both functions are unique SL-invariants of degree $k$. As a polynomial on tensors, this smallest degree invariant $\delta \in \mathbb{C}[(\mathbb{C}^k)^{\otimes d}]^{\mathrm{SL}(n)^{\times d}}_{k}$ was introduced by Cayley \cite{cay, cay2}, and the form $\omega$ can be viewed as its dual.   

\vspace{0.5em}

The duality for $\Delta$ and $\nabla$ gives conceptual interpretation to power expansions of the Cayley form and first hyperdeterminant. We denote $\omega^n = \omega \wedge \cdots \wedge \omega$ ($n$ times) and $\delta^n = \delta \vee \cdots \vee \delta$ ($n$ times). 

\vspace{0.5em}

\begin{corollary}[Power expansions of the Cayley form and first hyperdeterminant]
~ 

(i) For odd $d$ we have 
$$
\omega^n = \sum_{T \in \mathsf{T}_{01}((k \times n)^d)} \Delta_T(\mathsf{I}_n) \bigwedge e_T, \quad n \le k^{d-1}.
$$
In particular, for $n = k^{d-1}$ we have 
$
\omega^{k^{d-1}} = \pm \mathrm{AT}_d(k) \bigwedge e_{[k]^d}.
$

(ii) For even $d$ we have 
    $$
        \delta^n = \sum_{S \in \mathsf{T}_{\mathbb{N}}((k \times n)^d)} \Delta_{S}(\mathsf{I}_n) \bigvee 
        e_S, \quad n \ge 1. \quad
    $$
    In particular, for $n = k^{d-1}$ we have 
    $
    \langle \delta^{k^{d-1}}, \bigvee e_{[k]^d} \rangle = \pm \mathrm{AT}_d(k),
    $
    i.e. the coefficient of $\delta^{k^{d-1}}$ 
    at $\bigvee e_{[k]^d}$ is equal to $\pm \mathrm{AT}_d(k)$.

\end{corollary}

Let us note that the part (i) of this result about expansions of $\omega$ was proved by us in \cite{ayfund} using rather long  combinatorial argument about signs of Latin hypercubes. Here we give a conceptual explanation to this expansion. Moreover, this expansion is important for showing positivity of Kronecker coefficients. Namely, we proved in \cite{ayfund} that for odd $d \ge 3$ the condition $\mathrm{AT}_d(k) \ne 0$ (which conjecturally holds for all even $k$) 
implies the positivity of rectangular Kronecker coefficients $g(n \times k, \ldots, n \times k) > 0$ for all $n \le k^{d-1}$. Moreover, as we studied in \cite{ayuni}, the form $\omega$ plays an important role for understanding  unimodality properties of Kronecker coefficients.



\section{Preliminaries}
We denote $[n] := \{ 1, \ldots, n\}$ and $\{e_{i}\}_{i=1}^{n}$ is the standard basis of $\mathbb{C}^{n}$. For a composition $\alpha = (\alpha_1, \ldots, \alpha_{\ell})$ we denote $|\alpha| = \sum_{i} \alpha_i$ and $\ell(\alpha) = \ell$ its length. We denote the standard basis of the tensor space $(\mathbb{C}^n)^{\otimes d}$ by $e_{\mathbf{i}} = e_{i_1}\otimes\ldots\otimes e_{i_d}$ for $\mathbf{i} = (i_1,\ldots,i_d) \in [n]^d$.

\subsection{Partitions} A {\it partition} is a sequence $\lambda = (\lambda_{1}, \ldots, \lambda_{\ell})$ of positive integers $\lambda_1 \ge \cdots \ge \lambda_{\ell}$, where $\ell(\lambda) = \ell$ is its {\it length}. The {\it size} of $\lambda$ is $|\lambda| = \lambda_1 + \ldots + \lambda_{\ell}$. Every partition $\lambda$ can be represented as the {\it Young diagram} $D(\lambda):=\{(i,j)  : i \in [1, \ell], j\in [1,\lambda_{i}]\}$. We denote by $\lambda'$ the {\it conjugate} partition of $\lambda$ whose diagram is transposed. For a partition $\mu = (\mu_1,\ldots,\mu_p)$ we also define: the partition $\lambda + \mu := (\lambda_1 + \mu_1,\ldots,\lambda_q + \mu_q)$, where $q = \max(\ell, p)$; the sequence $-\lambda := (-\lambda_1,\ldots,-\lambda_\ell)$; and the sequence $\lambda - \mu := \lambda + \mathrm{rev}(-\mu)$, where $\mathrm{rev}(\mu) := (\mu_p,\ldots,\mu_1)$. We extend all operations to $d$-tuples of partitions (sequences) coordinate-wise. 

\subsection{Hypermatrices}

\textit{A $d$-dimensional hypermatrix} is 
an array $H = (H_{i_1,\ldots,i_d})_{i_1, \ldots, i_d \ge 1}$. 
Such hypermatrix displays coordinates of the tensor $H = \sum_{i_1,\ldots,i_d \ge 1} H_{i_1,\ldots,i_d} e_{i_1,\ldots,i_d} \in (\mathbb{C}^n)^{\otimes d}$ and each tensor can be represented as a hypermatrix, once basis is fixed. 

An integer-valued hypermatrix $H$ is called \textit{$\mathbb{N}$-hypermatrix} if it has \textit{finite support} $$\mathrm{supp}(H) = \{ (i_1,\ldots,i_d): H_{i_1,\ldots,i_d} \neq 0 \}$$ and $H_{i_1,\ldots,i_d} \in \mathbb{Z}_{\ge 0}$. If additionally $H_{i_1,\ldots,i_d} \in \{0, 1\}$, then $H$ is called \textit{$(0,1)$-hypermatrix}. 
The \textit{marginals} of the hypermatrix $H$ refer to the $d$-tuple of vectors $\pmb{\lambda} = (\lambda^{(1)},\ldots,\lambda^{(d)})$ with $\lambda^{(\ell)} = (\lambda^{(\ell)}_1, \ldots )$ given by 
$$
\lambda^{(\ell)}_j = \sum_{i_1,\ldots,\widehat{i_\ell},\ldots,i_d} H_{i_1,\ldots,j, \ldots, i_d}
$$
for each $\ell \in [d]$ and $j \ge 1$. In other words, $\lambda^{(\ell)}_j$ is the sum of elements in $j$-th {\it slice} of $\ell$-th direction.
We denote the sets of $\mathbb{N}$- and $(0,1)$-hypermatrices with marginals $\pmb{\lambda} = (\lambda^{(1)},\ldots,\lambda^{(d)})$ by $\mathsf{T}_{\mathbb{N}}(\pmb{\lambda})$ and $\mathsf{T}_{01}(\pmb{\lambda})$, respectively. 

\newcounter{l}
\newcounter{L}
\newcommand{\drawsquare}[4] {
    \foreach \y in {1*\value{l}+#2,2*\value{l}+#2, 3*\value{l}+#2} {
        \foreach \z in {1*\value{l}+#3, 2*\value{l}+#3, 3*\value{l}+#3} {
            \fillcell{#1}{\y}{\z}{#4};
        }
    }
}
\newcommand{\fillcell}[4] {
    \draw[#4] (#1,#2,#3) -- (#1,#2+\value{l},#3) -- (#1,#2+\value{l},#3+\value{l}) -- (#1,#2,#3+\value{l}) -- cycle;
}
\newcommand{\fillcelly}[4] {
    \draw[#4] (#1,#2,#3) -- (#1+\value{l},#2,#3) -- (#1+\value{l},#2,#3+\value{l}) -- (#1,#2,#3+\value{l}) -- cycle;
}
\newcommand{\fillcellz}[4] {
    \draw[#4] (#1,#2,#3) -- (#1+\value{l},#2,#3) -- (#1+\value{l},#2+\value{l},#3) -- (#1,#2+\value{l},#3) -- cycle;
}
\newcommand{\putcell}[5] {
    \filldraw[#4] (#1, #2+\value{l}/2, #3+\value{l}/2) node{#5};
}

\begin{figure}
	\setcounter{l}{1}
    \setcounter{L}{2*\value{l}}
    \begin{center}
        \begin{minipage}{0.4\textwidth} 
            \centering
            \resizebox{\textwidth}{!}{
                \picturematrix
            }
        \end{minipage}
        \begin{minipage}{0.55\textwidth} 
            \centering
            \begin{align*}
                \xrightarrow{
                    \quad 3-\text{line notation} \quad
                    }\quad
                \begin{pmatrix}
                    1111112222\\
                    1122331112\\
                    1133122221
                \end{pmatrix}\\
            \end{align*}
        \end{minipage}        
	\end{center}
    \caption{A 3-d hypermatrix $T$ with marginals $(64,532,442)$ and its $3$-line notation. Empty cells regarded as equal to $0$.}
    \label{fig:d-line}
\end{figure}

\subsection{$d$-line notation} For a $d$-dimensional $\mathbb{N}$-hypermatrix $H$, the \textit{$d$-line notation} of $H$ is the $d \times m$ table, where $m = |H| =  \sum_{\mathbf{i}} H_{\mathbf{i}}$, such that each column $(i_1,\ldots, i_d)^{}$ is repeated exactly $H_{i_1,\ldots,i_d}$ times, and
the columns are ordered lexicographically. 
We denote the resulting table by $d\text{-line}(H)$, see Fig.~\ref{fig:d-line}.

\subsection{Words}
Let $\alpha = (\alpha_1,\ldots,\alpha_\ell)$ be a sequence with $|\alpha| = m$. 

Let $A(\alpha)$ be the set of permutations of the word $1^{\alpha_1}\ldots \ell^{\alpha_\ell}$ of length $m$. The elements of $A(\alpha)$ are called the \textit{words of weight} $\alpha$. The positions of letters of a word $s \in A(\alpha)$ are called \textit{blocks}. For example, 
$$
    \text{for }s = 112113323 \in A(4,2,3) \text{ the blocks are }\{1,2,4,5\},\{3, 8\},\{6, 7, 9\}.
$$

A word $w\in A(\alpha)$ is called \textit{lattice}, if in each prefix there are at least as many letters $i$ as letters $i+1$, for any $i$. Note that the weight of a lattice word is always a partition. The set of lattice words of weight $\alpha$ is denoted by 
$A^{+}(\alpha) \subseteq A(\alpha)$. For example, $1123231$ is a lattice word of weight $322$, while $1232131$ is a word of weight $322$ which is not lattice.

For a $d$-tuple of weights $\pmb{\alpha} = (\alpha^{(1)},\ldots,\alpha^{(d)})$, with each $|\alpha^{(i)}| = m$, we denote by
$$A(\pmb{\alpha}) = A(\alpha^{(1)}) \times\ldots\times A(\alpha^{(d)})$$
the set of $d$-tuples of words of weight $\pmb{\alpha}$ (the set $A^+(\pmb{\alpha})$ is defined similarly).

\subsection{Tables and maps}\label{subsec:convention} Throughout the paper, a $d$-tuple of words $S = (s_1,\ldots,s_d) \in A(\pmb{\alpha})$ is interpreted in several equivalent ways:
\begin{enumerate}
    \item[(a)] $S$ is the $d \times m$ \textit{table of weight} $\pmb{\alpha}$ with the rows $s_1,\ldots,s_d$. 
    \item[(b)] $S$ is \textit{the map} $[m] \to \mathbb{N}^d$ with $S(j) = (a^{(1)}_j,\ldots,a^{(d)}_j)$ representing $j$-th column of the table (a);
    \item[(c)] $S$ is the hypermatrix $(S_{i_1,\ldots, i_d})$, so that $S_{i_1,\ldots,i_d}$ is the number of columns $(i_1,\ldots,i_d)$ in the table (a).
\end{enumerate}
While the tables (a) and maps (b) are essentially the same, the hypermatrices (c) will only be used when $S$ is considered modulo $S_m$ action.
We have the following presentations for the $d \times m$ table $S$ described in (a), (b): 
\begin{align*}
    S = \begin{pmatrix}
        \quad s_1 \quad  \\
        \quad \vdots \quad \\
        \quad  s_d \quad
    \end{pmatrix}
    = 
    \begin{pmatrix}
        \\    
        S(1) \ S(2) \ \cdots \ S(m)\\   
        \\   
    \end{pmatrix}.
\end{align*}
    

Let $B(\pmb{\alpha}) \subseteq A(\pmb{\alpha})$ be the subset of tables of weight $\pmb{\alpha}$ \textit{without repeating columns}, i.e. $S \in B(\pmb{\alpha})$ whenever $|\{S(1),\ldots, S(m)\}| = m$. Also, throught the paper we will use notation
$$
    ||S|| := \prod_{\mathbf{i} \in \mathbb{N}^d} S_{\mathbf{i}}! = |\mathrm{stab}_{S_m}(S)|,
$$
i.e. table $S$ stabilizer size w.r.t. column permutation group $S_m$.

\subsection{Lex-tables}
It is not hard to see that the $d$-line notation of hypermatrices from $\mathsf{T}_\mathbb{N}(\pmb{\lambda})$ corresponds to the tables of weight $\pmb{\lambda}$ with columns ordered lexicographically; we will shortly call them \textit{lex-tables} denoted as follows:
$$
    A^\mathrm{lex}(\pmb{\lambda}) := d\text{-line}(\mathsf{T}_\mathbb{N}(\pmb{\lambda})) \subseteq A(\pmb{\lambda}), 
    \quad 
    B^\mathrm{lex}(\pmb{\lambda}) :=  d\text{-line}(\mathsf{T}_{01}(\pmb{\lambda})) \subseteq B(\pmb{\lambda}).
$$
Alternatively, $A^\mathrm{lex}(\pmb{\lambda}) = A(\pmb{\lambda}) / S_m$, i.e. orbits of column permutation action so that each orbit is indexed with lexicographically minimal element.

Whenever we consider tables for arbitrary weights $\pmb{\lambda}$ of partitions of $m$, we write 
$$
    A_d(m) := \bigcup_{\pmb{\lambda} \vdash m} A(\pmb{\lambda}), \quad
    A_d := \bigcup_{m \ge 0} A_d(m), \qquad
    B_d(m) := \bigcup_{\pmb{\lambda} \vdash m} B(\pmb{\lambda}) \quad
    B_d := \bigcup_{m \ge 0} B_d(m)
$$
to indicate the sets of $d \times m$ tables without weight restriction. Similar to hypermatrices, the elements of $A_d$ and $B_d$ are called $\mathbb{N}$- and $(0,1)$-tables.

\subsection{Highest weight vectors} 
All the groups we consider are over $\mathbb{C}$.
The group $\mathrm{GL}(n)$ acts diagonally on $\otimes^m \mathbb{C}^n$ by left multiplication. We fix the standard basis of $\mathbb{C}^n = \langle e_1,\ldots,e_n \rangle$. The standard torus $T(n) := (\mathbb{C}^{\times})^{n} \subseteq \mathrm{GL}(n)$, the set of non-degenerate diagonal matrices, gives rise to the weight decomposition
$$
	\bigotimes^m \mathbb{C}^n = \bigoplus_{\alpha \in \mathbb{Z}^{n}_{\ge 0}} W_{\alpha}
$$
where $\alpha \in \mathbb{Z}^{n}_{\ge 0}$ and $W_{\alpha} = \{ v \in \bigotimes^m \mathbb{C}^{n} \mid t \cdot v = t^{\alpha} v = t_{1}^{\alpha_{1}}\ldots t_{\ell}^{\alpha_{\ell}} v\}$ is the space of weight  $\alpha$ vectors. 
The {\it highest weight vector} of weight $\lambda$ is a nonzero weight $\lambda$ vector which is invariant under the action of the subgroup of unitriangular matrices $U(n) \subset \mathrm{GL}(n)^{\times d}$, i.e. $U(n) \cdot v = v$ and $v \in W_{\lambda}$.

For the $n^{d}$-dimensional vector space $(\mathbb{C}^{n})^{\otimes d} = \langle e_{i_1,\ldots,i_d} \rangle_{1 \le i_j\le n}$ (where $e_{i_1,\ldots,i_d} = e_{i_1}\otimes\ldots\otimes e_{i_d}$), the natural (multilinear) action of the group $\mathrm{GL}(n)^{\times d}$ is given by
$$
    (G_1,\ldots,G_d) \cdot v_1 \otimes\ldots\otimes v_d = G_1 v_1 \otimes\ldots\otimes G_d v_d, 
$$
where $v_i \in \mathbb{C}^n$ and $G_i \in \mathrm{GL}(n)$. This action induces the action on the tensor space $\bigotimes^{m} (\mathbb{C}^{n})^{\otimes d}$. 
For the group $\mathrm{GL}(n)^{\times d}$ the definitions above can be rewritten similarly, where $T(n)$ is replaced with $T(n)^{\times d}$, $U(n)$ with $U(n)^{\times d}$ and the weights $\alpha \in \mathbb{Z}^n_{\ge 0}$ with $\pmb{\alpha} \in (\mathbb{Z}^n_{\ge 0})^{\times d}$.
We denote by $\mathrm{WV}_{\pmb{\lambda}} V$ and $\mathrm{HWV}_{\pmb{\lambda}} V$ the weight and the highest weight subspace of $V$ of weight $\pmb{\lambda}$.


\subsection{Representations of the symmetric and general linear groups} 
The standard references for the following material are \cite{fh, ful}.

For a partition $\lambda \vdash m$ the bijective map $Y: [m] \to D(\lambda)$ is called a \textit{Young tableau} of shape $\lambda$. In other words, $Y$ is the filling of the Young diagram $D(\lambda)$ with numbers $1,\ldots,m$. By $Y_{i,j} := Y^{-1}(i,j)$ we denote the number written in the box $(i,j)$.  Let $\mathrm{YT}(\lambda)$ be the set of  Young tableaux of shape $\lambda$. 

We call $Y$ a \textit{standard Young tableau} whenever $Y_{i,j} < \min(Y_{i+1,j}, Y_{i,j+1})$ (where we assume that $Y_{i,j} = \infty$ whenever $(i,j) \not\in D(\lambda)$). The set of standard Young tableaux is denoted by $\mathrm{SYT}(\lambda)$.

The irreducible representations of the group $\mathrm{GL}(n)$ are indexed by partitions $\lambda$, known as \textit{Weyl modules} denoted by $V(\lambda)$. It is known that the weight space of weight $\lambda$ is $1$-dimensional in $V(\lambda)$ called \textit{the highest weight line}. This line characterizes the irreducible representation.

The irreducible representations of the symmetric group $S_m$ are indexed by partitions $\lambda$, known as  \textit{Specht modules} denoted by $[\lambda]$. 

Let us recall the classical construction of irreducible representations of the symmetric group $S_{m}$. The group $S_{m}$ acts on tableaux of shape $\lambda$ by $(\pi T)_{i,j}=\pi(T_{i,j})$, and also acts on the tensor space $\bigotimes^{m} \mathbb{C}^n$ by permuting the tensor factors.
Let $\lambda \vdash m$ be a partition and $T \in \mathrm{SYT}(\lambda)$. Let $R(T)$ and $C(T)$ be the row and column stabilizers of $T$, and define the following elements: 
\begin{align*}
    a_{T} = \sum_{g \in R(T)}g, \qquad b_{T} = \sum_{g \in R(T)}\sgn(g) g, \qquad c_T = b_T a_T.
\end{align*}
It is known that $c_{T}$ is a projector and
$$
c_{T}: \bigotimes^{m}\mathbb{C}^{n} \to V(\lambda)_{T}, \qquad c_{T}: \mathbb{C}[S_{m}] \to [\lambda]_T
$$
where the subscript $T$ of $V(\lambda)_{T}$ indicates concrete irreducible copy of $\bigotimes^m \mathbb{C}^n$.

\subsection{Symmetrization and alternation}
Let $V$ be a vector space and $\bigotimes V = \bigoplus_{m \ge 0} \bigotimes^m V$ be its graded tensor algebra. The $m$-th grade of the tensor algebra $\bigotimes^m V$ has two important subspaces: symmetric and alternating spaces.

The grade-preserving maps $\mathrm{Sym}: \bigotimes V \to \bigotimes V$ and $\mathrm{Alt}: \bigotimes V \to \bigotimes V$ are defined as follows:
\begin{align*}
\mathrm{Sym}(v_{1}\otimes\ldots\otimes v_{m}) &:= \sum_{\omega \in S_{m}} v_{\omega(1)}\otimes\ldots\otimes v_{\omega(m)},\\
	\mathrm{Alt}(v_{1}\otimes\ldots\otimes v_{m}) &:= \sum_{\omega \in S_{m}} \sgn(\omega)\, v_{\omega(1)}\otimes\ldots\otimes v_{\omega(m)}
\end{align*}
and extended linearly.\footnote{These maps are usually defined with a scalar factor $1/m!$} These operators are scalar multiples of projectors (i.e. $P^2 = \alpha P$ for $\alpha \neq 0$).
Denote the images $\bigvee^m V := \mathrm{Sym}(\bigotimes^m V)$ 
and $\bigwedge^{m} V := \mathrm{Alt}(\bigotimes^m V)$ for each $m \ge 0$ and by
$$
\bigvee V := \bigoplus_{m\ge 0} \bigvee^{m} V, \quad \bigwedge V := \bigoplus_{m \ge 0} \bigwedge^m V
$$
the so-called \textit{symmetric} and \textit{skew-symmetric} algebras. To make these vector spaces algebras one can endow them with the products: $a \vee b := \mathrm{Sym}(a \otimes b)$ and $a \wedge b := \mathrm{Alt}(a \otimes b)$ defined on homogeneous elements. The symmetrc algebra $\bigvee V$ with the given product is isomorphic to the ring of polynomials of $\dim V$ variables, i.e. $\mathbb{C}[V] \cong \bigvee V$. The alternating algebra $\bigwedge V$ is also referred to as exterior algebra or wedge space. 

\vspace{0.5em}

Note that for $V = (\mathbb{C}^{n})^{\otimes d}$, the sets $\{\bigvee e_{S} : S \in \mathsf{T}_{\mathbb{N}}(\pmb{\lambda}) \}$ and $\{\bigwedge e_{T} : T \in \mathsf{T}_{01}(\pmb{\lambda}') \}$ are the standard bases of the weight spaces $\mathrm{WV}_{\pmb{\lambda}} \bigvee V$ and $\mathrm{WV}_{\pmb{\lambda}} \bigwedge V$, respectively.

\subsection{Kronecker coefficients}
The dimensions of the highest weight subspaces that we consider are given by generalized Kronecker coefficients. The \textit{generalized Kronecker coefficient} $g(\pmb{\lambda})$, indexed by a $d$-tuple $\pmb{\lambda}$  of partitions of $m$, is the multiplicity of the trivial $S_m$-irreducible representation $[1 \times m]$ in the tensor product $[\lambda^{(1)}]\otimes \ldots \otimes [\lambda^{(d)}]$ of $S_m$-irreducibles (tensor product is again a representation with diagonal $S_m$-action). Alternatively, $g(\pmb{\lambda})$ is the multipcity of $[\lambda^{(1)}]$ in $[\lambda^{(2)}]\otimes \ldots \otimes [\lambda^{(d)}]$.
It is known, that $[\lambda]\otimes[\mu] \cong [\lambda']\otimes[\mu']$, which, together with the above, implies the following properties:
    \begin{enumerate}
        \item[(a)] $g(\lambda^{(1)},\ldots,\lambda^{(d)}) = g(\lambda^{(\sigma(1))},\ldots,\lambda^{(\sigma(d))})$ for any permutation $\sigma \in S_d$.
        \item[(b)] $g(\lambda^{(1)},\lambda^{(2)},\ldots,\lambda^{(d)}) = g((\lambda^{(1)})',(\lambda^{(2)})', \lambda^{(3)}, \ldots,\lambda^{(d)})$.
        \item[(c)] $g(1\times m, \pmb{\lambda}) = g(\pmb{\lambda})$.
    \end{enumerate} 
For $d = 2$, we have $g(\lambda, \mu) = \delta_{\lambda, \mu}$ is the Kronecker delta and for $d = 3$ we recover the usual Kronecker coefficients. 
Then we have
$$
    \dim \mathrm{HWV}_{\pmb{\lambda}} \bigvee (\mathbb{C}^n)^{\otimes d} = g(1 \times m, \pmb{\lambda}) = g(\pmb{\lambda}), \quad \text{for any }d \ge 2,
$$
and
$$
\dim \mathrm{HWV}_{\pmb{\lambda}} \bigwedge (\mathbb{C}^n)^{\otimes d} = g(m \times 1, \pmb{\lambda}) = g(1 \times m, (\lambda^{(1)})', \lambda^{(2)},\ldots,\lambda^{(d)}),
$$
where the latter is equal to $g((\lambda^{(1)})', \lambda^{(2)},\ldots,\lambda^{(d)})$, which is rewrites to $g(\pmb{\lambda}')$ for odd $d$.
\begin{remark}
    The Kronecker coefficients can also be defined via representations of $\mathrm{GL}(n)$.
    Consider $\mathrm{GL}(n^d)$-irreducible representation $V_{n^d}(1 \times m)$ (the subscript emphasizes the group) of degree $m$ polynomials. Under the group homomorphism $\mathrm{GL}(n)^{\times d} \to \mathrm{GL}(n^d)$, given by $g_1\times\ldots\times g_d \to g_1 \boxtimes \ldots \boxtimes g_d$ (where $\boxtimes$ is Kronecker product of matrices), this representation restricts to the representation of its subgroup $\mathrm{GL}(n)^{\times d}$ and $g(\pmb{\lambda})$ is the multiplicity of irreducible $V_n(\lambda^{(1)})\otimes\ldots\otimes V_n(\lambda^{(d)})$ in this decomposition.
\end{remark}


\section{Highest weight vectors for single tensor component}\label{sec:spanningset}
In this section we describe how to construct a basis of the space $\mathrm{HWV}_{{\lambda}} \bigotimes^m\mathbb{C}^n$, i.e. when $d = 1$.
Subsequently, we will generalize this to $d$ tensor components and then project to symmetric and alternating spaces to obtain spanning sets of these spaces. The idea of construction is not new, see e.g. \cite{bi, widg, widg1}; 
here we present a self-contained description. 

\begin{definition}[Signature function]\label{def:sign}
    Let $\lambda = (\lambda_1,\ldots,\lambda_\ell) \vdash m$ be a partition and $s \in A(\lambda)$ be a word of weight $\lambda$. By $\mathrm{block}_i = s^{-1}(i)$ denote the set of positions of the letter $i$ in $s$, referred to as \textit{blocks}. For a map $\sigma: [m] \to \mathbb{N}$, define the following {\it signature} function:
    \begin{equation}
        \sgn_s(\sigma) := \sgn(\sigma(\mathrm{block}_1))\cdots\sgn(\sigma(\mathrm{block}_\ell))
    \end{equation}
    where for a set $B = \{b_1<\ldots< b_k \}$ we denote $\sigma(B) = (\sigma(b_1),\ldots,\sigma(b_k))$ and
    \begin{equation}
        \sgn(a_1,\ldots, a_k) = 
        \begin{cases}
            \sgn(a),& \text{if } a = (a_1, \ldots, a_k) \text{ is a permutation of $[k]$}\\
            0,              & \text{otherwise}.
        \end{cases}
    \end{equation}
  
    For example, for $s = 11\textcolor{blue}{22}\textcolor{orange}{3}11\textcolor{orange}{3}\textcolor{blue}{2}$ of weight $(4,3,2)$ we have 
    \begin{align*}
        \sgn_s(12\textcolor{blue}{12}\textcolor{orange}{1}34\textcolor{orange}{2}\textcolor{blue}{3}) &= \sgn(1234)\, \sgn(\textcolor{blue}{123})\, \sgn(\textcolor{orange}{12}) = +1,\\
        \sgn_s(23\textcolor{blue}{32}\textcolor{orange}{1}14\textcolor{orange}{2}\textcolor{blue}{1}) &= \sgn(2314)\, \sgn(\textcolor{blue}{321})\, \sgn(\textcolor{orange}{12}) = -1,\\
        \sgn_s(24\textcolor{blue}{35}\textcolor{orange}{1}14\textcolor{orange}{2}\textcolor{blue}{1}) &= \underbrace{\sgn(2414)}_{= 0 }\, \underbrace{\sgn(\textcolor{blue}{351})}_{=0}\, \sgn(\textcolor{orange}{12}) = 0. 
    \end{align*}
    In other words, $\sgn_s(w)$ is zero whenever the word $w$ does not have permutations in blocks of $s$; 
    otherwise, its value is equal to the product of block permutation signatures. It is easy to see that only words of weight $\lambda'$ can be non-zero at $\sgn_s(\cdot)$ for $s \in A(\lambda)$.
\end{definition}

We first obtain a spanning set of $\mathrm{HWV} \otimes^m \mathbb{C}^n$. 

\begin{remark}
The construction we show differs from the one represented in \cite{widg1} by reducing the size of the index set from $S_m$ to $A(\lambda)$.
\end{remark}

\begin{definition}[HW $p$-vectors]

    Define $\ell$-tensor $\dett_{\ell} \in \bigotimes^{\ell} \mathbb{C}^n$ for $\ell \le m$ as the (dual of) determinant as follows:
    $$
        \dett_{\ell} := \sum_{\pi \in S_\ell} \sgn(\pi)\, e_{\pi(1)} \otimes \ldots \otimes e_{\pi(\ell)},
    $$
    i.e. $\dett^{*}(v_1,\ldots,v_\ell)$ is the determinant of the top $\ell \times \ell$ submatrix of an $m \times \ell$ matrix with the vectors $v_i$ given by columns. It is not hard to verify that $\det_{\ell}$ is the highest weight vector of weight $(1^{\ell},0^{m - \ell})$. 

    Let $\lambda \vdash m$ be a partition with $\ell(\lambda) \le n$ and the first part $k = \lambda_1$. For a word $t \in A(\lambda')$ of conjugate weight $\lambda'$ define
    $$
        p_t := \pi_t \cdot (\dett_{\lambda'_1} \otimes\ldots\otimes \dett_{\lambda'_k}) \in \mathrm{HWV}_\lambda \bigotimes^m \mathbb{C}^n,
    $$
    where $\pi_t \in S_m$ is the permutation of minimal length which sends the word $1^{\lambda'_1}\ldots k^{\lambda'_k}$ to the word $t$ (the minimal element of the coset in $S_m / Y_{\lambda'}$ corresponding to $t$), i.e. it does not permute entries of the word within the same block of the initial word. Single $t$-block represents positions where $\det$ is taken from. For example, $p_{11122} = \dett_3 \otimes \dett_2$ and $p_{11212} = (35)\cdot (\dett_3 \otimes \dett_2)$. 

    
    We will work with concrete form of $p_t$ for $t \in A(\lambda')$ via $\sgn_t(\cdot)$ function:
    $$
        p_t = \sum_{s \in A(\lambda)} \sgn_t(s) \cdot \bigotimes_{i=1}^m e_{s(i)} \in \mathrm{HWV}_{\lambda} \bigotimes^m \mathbb{C}^n.
    $$
    Since $\det_\ell \in \mathrm{HWV}_{(1^\ell)}\otimes^\ell \mathbb{C}^n$, then $p_t \in \mathrm{HWV}_{\lambda}\otimes^m \mathbb{C}^n$ for any $t \in A(\lambda')$. 

    For a tableau $T \in \mathrm{YT}(\lambda)$ let $col(T) \in A(\lambda')$ be a word $w_1\ldots w_m$ with $w_i$ equal to the index of the column containing letter $i$ in $T$, and $row(T) \in A(\lambda)$ is defined similarly for rows.
    For instance, for
    \begin{align}\label{ex:tableauexample}
    \ytableausetup{centertableaux}
    T = \begin{ytableau}
    1 & 3 & 5 \\
    2 & 6 & 7 \\
    4
    \end{ytableau}
    \end{align}
    we have $col(T) = 1121323$ and $row(T) = 1213122$. Note that the maps 
    $$
    row:\mathrm{YT}(\lambda) \to A(\lambda)\quad\text{and}\quad col:\mathrm{YT}(\lambda) \to A(\lambda')
    $$ 
    are surjections. Meanwhile, the maps between SYTs and lattice words
    $$
    row: \mathrm{SYT}(\lambda) \to A^+(\lambda)\quad\text{and}\quad col: \mathrm{SYT}(\lambda) \to A^+(\lambda')
    $$
    are bijections. We write $p_{T} := p_{col(T)}$. For instance, for $T$ in \eqref{ex:tableauexample} we have $p_T = p_{1121323}$.
\end{definition}

\begin{remark}
    Notably, the original work \cite{young} of Alfred Young on representation theory of the symmetric group contain analogous constructions; in particular, he considered pairs of `good' tableaux, which in essence are the pairs $(X \in row^{-1}(s), Y \in col^{-1}(t))$ with $\sgn_s(t) \neq 0$, see also \cite{garsia}. We will see later that the introduced signature functions are related to double cosets of Young subgroups.
\end{remark}

The following proposition is obtained by standard well-known technique. 

\begin{proposition}\label{th:basis-single}
    The set $\{p_{col(T)}\}_{T \in \mathrm{SYT}(\lambda)} = \{p_{t}\}_{t \in A^+(\lambda)}$ forms the basis of $\mathrm{HWV}_{\lambda} \bigotimes^m \mathbb{C}^n$.
\end{proposition}
\begin{proof}
    By the Schur--Weyl duality we have the following decomposition with respect to the action of $S_m \times \mathrm{GL}(n)$: 
    $$
        (\mathbb{C}^n)^{\otimes m} = \bigoplus_{\lambda \vdash M,\, \ell(\lambda) \le m} [\lambda] \otimes V(\lambda).
    $$
    Moreover, inducing the action to $\mathrm{GL}(n)$, for $T \in \mathrm{SYT}(\lambda)$ let $V(\lambda)_T := c_T((\mathbb{C}^n)^{\otimes m})$ be the image of Young symmetrizer, which is $\mathrm{GL}(n)$-irreducibe. 
    Then each irreducible 
    $V(\lambda)_{T}$ 
    contains unique highest weight line $L_{T}$. On the other hand,
    \begin{align*}
        \mathrm{HWV}_{\lambda}(\mathbb{C}^n)^{\otimes m} =  
        \mathrm{HWV}_{\lambda} \bigoplus_{T \in \mathrm{SYT}(\lambda)} V(\lambda)_T = 
        \bigoplus_{T \in \mathrm{SYT}(\lambda)} L_{T}.
    \end{align*}
    In particular, the highest weight space is an irreducible $S_{m}$-representation indexed by $\lambda$. Set $W_{\lambda} = \mathrm{WV}_{\lambda} \bigotimes^m \mathbb{C}^n$.
    We claim that $c_{T}(W_{\lambda}) = L_{T}$. This can be seen from the following diagram:
    \[\begin{tikzcd}
    W_{\lambda}  \arrow[swap]{d}{c_{T}} & \subseteq & (\mathbb{C}^n)^{\otimes m} \arrow{d}{c_{T}} \\
    L_{T} & \subseteq & V(\lambda)_T
    \end{tikzcd}
    \]
    and the fact that $c_{T}$ preserves weight.
    Since $c_{T}((\mathbb{C}^n)^{\otimes m}) = V(\lambda)_T$ and $V(\lambda)_T \cap W_{\lambda} = L_{T}$ we conclude that $c_{T}(W_{\lambda}) = L_{T}$. It remains to show that $L_T = p_{col(T)}$.

    Let $e_{row(T)} = e_{row(T)_1} \otimes\ldots\otimes e_{row(T)_m} \in W_\lambda$.
    Consider the product $c_{T}\cdot e_{row(T)} = b_{T} \cdot (a_{T} \cdot e_{row(T)})$. Note that for any $\pi \in R(T)$ (row stabilizer) there holds $\pi e_{row(T)} = e_{row(T)}$, since positions from the fixed row of $T$ are filled with the same letter in $row(T)$. Thus $a_{T}(e_{T}) = |R(T)| \cdot e_{T}$. Furthermore,  it turns out that $b_{T}(e_{row(T)}) = p_{col(T)}$. Indeed, $b_T$ alternates positions from the same column, thus $i$-th instance of each letter in $row(T)$ corresponds to positions alternated in the $i$-th column. In particular, $\langle p_{col(T)}, e_{row(T)} \rangle = 1$.
    Hence, $p_{col(T)}$ defines highest weight line of $V(\lambda)_{T}$ and the set $\{p_{col(T)}\}_{T \in \mathrm{SYT}(\lambda)}$ forms the basis for the corresponding highest weight space.
\end{proof}
\begin{remark}
    The space spanned by the forms $\{p_t\}$ are isomorphic to the both Specht module constructions: polytabloids and column tabloids \cite{ful}, for which we have $\pi \cdot p_t = \sgn(\pi) p_t$ only for $\pi \in \mathrm{Stab}(t)$ (column stabilizers).
    In the next section we will indirectly apply the basis rescaling $p_t \to (-1)^t p_t$ (see Def.~\ref{def:msign}), which suitably rescales the action of $S_m$: $\pi \cdot p_t = (-1)^{\pi} p_t$ for any $\pi \in S_m$. These signatures are important for Theorem~\ref{th:intro2}.
\end{remark}

\section{Spanning sets for tensor spaces}\label{sec:spanningtensor}
Let $\pmb{\lambda} = (\lambda^{(1)}, \ldots, \lambda^{(d)})$ with each $\lambda^{(i)} \vdash m$ and $\ell(\lambda^{(i)}) \le n$.
We now construct a basis of the highest weight space $\mathrm{HWV}_{\pmb{\lambda}} \bigotimes^m (\mathbb{C}^{n})^{\otimes d}$. The key observation is that there is a natural isomorphism
$$
    \bigotimes^m \bigotimes^d \mathbb{C}^n \cong \bigotimes^d \bigotimes^m  \mathbb{C}^n,
$$
just by changing the order of tensor product. The action of the groups $S_m$ and $\mathrm{GL}(n)^{\times d}$ can be described by the table view of tensor product:
\begin{center}
    \begin{tabular}{  c | l l l l l l l l} 
                \quad\ \ \ $S_m\acts$  & $1$            &         & $2$            &        &\ldots &        & $m$           & \\ \hline
$\mathrm{GL}(n)_1\acts$ & $\mathbb{C}^n$ & $\otimes$ & $\mathbb{C}^n$ &$\otimes$ &\ldots &$\otimes$ &$\mathbb{C}^n$ & $\otimes$ \\ 
\ldots                & \ldots       & \ldots  & \ldots       &\ldots  &\ldots &\ldots  &\ldots       & $\otimes$ \\
$\mathrm{GL}(n)_d\acts$ & $\mathbb{C}^n$ & $\otimes$ & $\mathbb{C}^n$ &$\otimes$ &\ldots &$\otimes$ &$\mathbb{C}^n$ & \\
\end{tabular}
\end{center}
where $S_m$ acts by permuting columns and each $\mathrm{GL}(n)$ acts diagonally on the corresponding row.
\begin{definition}
    Let $T = (t_1,\ldots,t_d) \in A(\pmb{\lambda'})$ be a $d \times m$ table with rows $t_i$. Define the following vector:
    $$
        P'_{T} := p_{t_1} \otimes \ldots \otimes p_{t_d} \in \mathrm{HWV}_{\pmb{\lambda}}\bigotimes^{d} \bigotimes^m \mathbb{C}^n
    $$
    and let $P_T$ denote its image in $\bigotimes^m \bigotimes^d \mathbb{C}^n$ after reordering tensor factors to $\bigotimes^d \bigotimes^m  \mathbb{C}^n$.
    We can write down the concrete form of $P_T$ as follows:
    \begin{align*}
        P'_{T}
        &= \left(
            \sum_{s_1 \in A(\lambda^{(1)})} \sgn_{t_1}(s_1) \bigotimes_{i=1}^m e_{s_1(i)}
        \right)
        \otimes\ldots\otimes 
        \left(
            \sum_{s_d \in A(\lambda^{(d)})} \sgn_{t_d}(s_d) \bigotimes_{i=1}^m e_{s_d(i)}
        \right)
            \\
        &= \sum_{(s_{1},\ldots,s_{d}) \in A(\pmb{\lambda})}
        \sgn_{t_1}(s_1)\cdots\sgn_{t_d}(s_d) 
        \bigotimes_{i=1}^m e_{s_1(i)} 
        \otimes \ldots \otimes 
        \bigotimes_{i=1}^m e_{s_d(i)}
        \\
        &\cong 
        \sum_{(s_{1},\ldots,s_{d}) \in A(\pmb{\lambda})}
        \sgn_{t_1}(s_1)\cdots\sgn_{t_d}(s_d) 
        \bigotimes_{i=1}^m 
            e_{s_1(i)} 
            \otimes \ldots \otimes 
            e_{s_d(i)} = P_T,
    \end{align*}
    and hence viewing $(s_1,\ldots,s_d) = S \in A(\pmb{\lambda})$ we rewrite
    \begin{align}
        P_T = \sum_{S \in A(\pmb{\lambda})} \sgn_T(S) \bigotimes_{i=1}^m e_{S(i)}, 
    \end{align}
    where $\sgn_T(S) = \sgn_{t_1}(s_1)\cdots\sgn_{t_d}(s_d)$. Note that in general $\sgn_T(S) \neq 0$ if and only if $\mathrm{weight}(T) = \mathrm{weight}(S)'$.
\end{definition}

\begin{proposition}\label{th:tensor-basis}
    The set $$\{P_{col(T)} \}_{T \in \mathrm{SYT}(\pmb{\lambda})} = \{P_{T}\}_{T \in A^+(\pmb{\lambda})}$$ forms a basis of $\mathrm{HWV}_{\pmb{\lambda}} \bigotimes^m (\mathbb{C}^n)^{\otimes d}$.
\end{proposition}
\begin{proof}
    Rearranging the tensor product we have
    $$
        \mathrm{HWV}^{\mathrm{GL}(n)^{\times d}}_{\pmb{\lambda}} \bigotimes^m \bigotimes^d \mathbb{C}^n \cong 
        \bigotimes_{i=1}^d \mathrm{HWV}^{\mathrm{GL}(n)}_{\lambda^{(i)}} \bigotimes^m (\mathbb{C}^n)^{\otimes d}
    $$
    and by Proposition~\ref{th:basis-single}, the tensors $P_{col(T)}$ for $T \in \mathrm{SYT}(\pmb{\lambda})$ form the basis of the desired space.
\end{proof}
In particular, the dimension of this highest weight space is equal to $f^{\lambda^{(1)}} \cdot\ldots\cdot f^{\lambda^{(d)}}$, where $f^{\lambda} = \dim [\lambda]$ which can be computed by the hook-length formula.

\subsection{Symmetric and alternating subspaces}

We now obtain a spanning set of the highest weight spaces of $\bigvee^m (\mathbb{C}^n)^{\otimes d}$ and $\bigwedge^m (\mathbb{C}^n)^{\otimes d}$ for arbitrary highest weight $\pmb{\lambda}$. 
\begin{definition}\label{def:msign}
    For a word $w = w_1\ldots w_n \in \mathbb{N}^n$ a pair $(i,j)$ is called \textit{inversion} if $i < j$ and $w_i > w_j$. Denote the \textit{number of inversions} of $w$ by $\mathrm{inv}(w)$ and the \textit{word signature} as $(-1)^w := (-1)^{\mathrm{inv}(w)}$. For a $d \times m$ table $T = (t_1,\ldots,t_d)$ define the sign 
    $$(-1)^T := (-1)^{t_1}\cdots(-1)^{t_d}.$$
We choose this notation to indicate that $(-1)^T = \pm 1$ is always nonzero, while $\sgn_T(\cdot)$ can be zero.
\end{definition}

We are now ready to define the highest weight vectors.

\begin{definition}[HWV polynomials and forms]\label{def:nabla-delta1}
    For a $d$-tuple $\pmb{\lambda}$ of partitions of $m$,
    define the vectors
    \begin{align} 
        \Delta_T &:= \frac{(-1)^T}{||T||}\sum_{X \in A(\pmb{\lambda})} \sgn_T(X) \bigvee e_X
        \in \mathrm{WV}_{\pmb{\lambda}} \bigvee^m (\mathbb{C}^n)^{\otimes d}, \quad T \in A(\pmb{\lambda'}), \label{eq:def-delta}
        \\  
        \nabla_S &:= \frac{(-1)^S}{||S||}\sum_{Y \in B(\pmb{\lambda})} \sgn_S(Y) \bigwedge e_Y 
        \in \mathrm{WV}_{\pmb{\lambda}} \bigwedge^m (\mathbb{C}^n)^{\otimes d}, \quad S \in A(\pmb{\lambda'}) \label{eq:def-nabla},
    \end{align}
    where $||X||:=|\mathrm{stab}_{S_m}(X)| = |\{\pi X = X \mid \pi \in S_m\}| = \prod_{\mathbf{i}}X_{\mathbf{i}}!$ denotes the size of stabilizer of the column permutation action of $S_m$.
    We also denote $\Deltaw_T := (-1)^T \Delta_T$ and $\nablaw_S := (-1)^S \nabla_S$ the vectors without sign factors in the sums above. The reason for that choice of the sign factor is in  Lemma~\ref{lemma:colswap}.
    The stabilizer factor $||X||$ and collinear terms in the sum will be justified and treated later in Section~\ref{sec:duality}.
    We will see that these vectors are highest weight vectors, since they are images of the projectors $\mathrm{Sym}(P_S) = \Deltaw_S \cdot ||S||$ and $\mathrm{Alt}(P_S) = \nablaw_S \cdot ||S||$.
\end{definition}
\begin{remark}
    These vectors can be written in a bit more readable, yet redundant form:
    \begin{align*}
        \Deltaw_T = \sum_{\sigma: [m] \to [n]^d} \sgn_T(\sigma) \bigvee_{i=1}^m e_{\sigma(i)}, 
        \quad 
        \nablaw_S = \sum_{\sigma: [m] \to [n]^d} \sgn_S(\sigma) \bigwedge_{i=1}^m e_{\sigma(i)}.
    \end{align*}
    In these sums, many terms vanish because of the sign functions. The map $\sigma$ can be regarded as a $d \times m$ table of weight $\pmb{\lambda}$, since other weights vanish. Originally, $\Deltaw_T$ polynomials were presented in this form in \cite{widg1}; we also studied them in \cite{ayfund}. 
\end{remark}
\begin{remark}
    For $d = 1$, the situation is trivial. For $t \in A(\lambda')$ polynomial $\Delta_{t} = 0$ unless $\lambda = (k)$ and $t = 12\ldots k \in A(1^k)$, in this case we have $\Delta_t = \vee (e_1)^k$ -- the only highest weight vector of $V(k)$. Similarly, $\nabla_s = 0$ unless $\lambda = 1^k$ and $s = 1^k \in A(k)$, in this case $\nabla_s = e_1 \wedge \ldots \wedge e_k \in A(1^k)$ -- the unique highest weight vector of irreducibe $V(1^k)$.
\end{remark}

\begin{theorem}[$\Delta, \nabla$ are spanning sets of HWV]\label{th:spanning-sets}
    Let $\pmb{\lambda}$ be a $d$-tuple of partitions of $m$ with lengths at most $n$. 

    (i) For odd $d$ we have 
    \begin{align*}
        \mathrm{HWV}_{\pmb{\lambda}} \bigvee (\mathbb{C}^n)^{\otimes d} 
        &= \mathrm{span}\{ 
            \Delta_T \mid T \in B^+(\pmb{\lambda'})
        \},\\
        \mathrm{HWV}_{\pmb{\lambda}} \bigwedge (\mathbb{C}^n)^{\otimes d} 
        &= \mathrm{span}\{ 
            \nabla_S \mid S \in A^+(\pmb{\lambda'})
        \}.
    \end{align*}

    (ii) For even $d$ we have
    \begin{align*}
        \mathrm{HWV}_{\pmb{\lambda}} \bigvee (\mathbb{C}^n)^{\otimes d} 
        &= \mathrm{span}\{ 
            \Delta_T \mid T \in A^+(\pmb{\lambda'})
        \},\\
        \mathrm{HWV}_{\pmb{\lambda}} \bigwedge (\mathbb{C}^n)^{\otimes d} 
        &= \mathrm{span}\{ 
            \nabla_S \mid S \in B^+(\pmb{\lambda'})
        \}.
    \end{align*}
\end{theorem}

We shall first prove some lemmas. Recall that $S_m$ acts on $d \times m$ tables from $A(\pmb{\lambda})$ by permuting the columns. The following lemma is important and will be useful throughout the paper.

\begin{lemma}[Column swap rule]\label{lemma:colswap}
    Let $\pmb{\lambda}$ be a $d$-tuple of partitions of $m$ and $S, T \in A(\pmb{\lambda'})$. Then for any $\pi \in S_m$ we have:
    \begin{align*}
        \Delta_{\pi T} = \sgn(\pi)^d\Delta_T \qquad\text{and}\qquad \nabla_{\pi S} = \sgn(\pi)^{d+1}\nabla_S.
    \end{align*}
\end{lemma}
\begin{proof}
    It is enough to show the result for a transposition $\alpha=(i,i+1)$ of $S_m$ for some $i\in[m-1]$. We bijectively map the term $\bigvee e_S$ of $\Delta_T$ in \eqref{eq:def-delta} to the term $\bigvee e_{\alpha S}$ of $\Delta_{\alpha T}$. Let us compare their coefficients and show that 
    $$
        (-1)^T\sgn_{T}(X) = (-1)^d(-1)^{\alpha T}\sgn_{\alpha T}(\alpha X).
    $$ 
    Each signature above is the product of signs over the rows.
    For a fixed $k$-th row $t_k$, the sign $\sgn_{\alpha t_k}(\alpha X)$ changes by $-1$ only whenever we change adjacent letters of the same block of $t_k$. On the other hand, $(-1)^{\alpha T}$ changes by $-1$ if distinct letters were swapped. Hence, the product changes by $-1$ in any case in each row, implying the total change by $(-1)^d$. The same holds for $\nabla_{\alpha S}$, but there is an extra sign change due to the property of the wedge product. Since $S_m$ is generated by transpositions of adjacent letters, the result follows.
\end{proof}

\begin{lemma}[Vanishing rule]\label{lemma:vanish}
    Let $T \in A(\pmb{\lambda'}) \backslash B(\pmb{\lambda'})$, i.e. $T$ has two equal columns. If $d$ is odd, then $\Delta_T = 0$, and if $d$ is even, then $\nabla_T = 0$.
\end{lemma}
\begin{proof}
    By Lemma~\ref{lemma:colswap} we can reorder columns freely, up to a sign factor. Let equal columns be adjacent, say the columns $1$ and $2$. On the one hand, $(1,2)T = T$. 
    On the other hand, by Lemma~\ref{lemma:colswap}, $\Delta_{(1,2)T} = (-1)^{d} \Delta_T = -\Delta_T$ for odd $d$ and $\nabla_{(1,2)T} = (-1)^{d+1} \nabla_T = -\nabla_T$ for even $d$, which implies the statement.
\end{proof}

\begin{proof}[Proof of Theorem~\ref{th:spanning-sets}]
    First, let us show that $\Delta_T$ and $\nabla_S$ are indeed highest weight vectors. By symmetrization of $P_T$ and alternation of $P_S$ we obtain that 
    \begin{align*}
        \mathrm{Sym}(P_T) 
        &= \sum_{X \in A(\lambda)} \sgn_T(X)\ \mathrm{Sym}
            \left(\bigotimes e_X\right)
        = \sum_{X \in A(\lambda)} \sgn_T(X) \bigvee e_X = \Deltaw_T \cdot ||T||,
        \\
        \mathrm{Alt}(P_S) 
        &= \sum_{Y \in A(\lambda)} \sgn_S(Y) \ \ \mathrm{Alt}\left(\bigotimes e_Y\right) \ \ 
        = \sum_{Y \in B(\lambda)} \sgn_S(Y) \bigwedge e_Y = \nablaw_S \cdot ||S||.
    \end{align*} 
    By Lemma~\ref{lemma:vanish}, $\Delta_R = 0$ for odd $d$ and $\nabla_R = 0$ for even $d$ whenever $R \in A(\pmb{\lambda'}) \backslash B(\pmb{\lambda'})$. Since the family of tensors $P_T$ (or $P_S$) span the highest weight subspace of the tensor space (Proposition~\ref{th:tensor-basis}),  the projection of this family (of highest weight tensors) to the corresponding subspace of polynomials (or forms) spans the highest weight subspace of polynomials (or forms) of the corresponding weight.
\end{proof}

Note that the sets $A^+(\pmb{\lambda})$ of lattice tables and $A^{\mathrm{lex}}(\pmb{\lambda}) = d\text{-line}(\mathsf{T}_\mathbb{N}(\pmb{\lambda}))$ of lex-tables are not subsets of one another. Using Lemma~\ref{lemma:colswap} we have the following. 

\begin{corollary}
    The statement of Theorem~\ref{th:spanning-sets} is true if the sets $A^+$ and $B^+$ are replaced with $A^\mathrm{lex}$ and $B^\mathrm{lex}$, i.e. if lattice tables are replaced with lex-tables.
\end{corollary}


\subsection{Some useful properties}

We now discuss several properties of these highest weight vectors and forms. 

\subsubsection{Internal symmetries}
Some pairs of tables produce the same polynomial or form, up to a constant. This happens due to the existence of some internal symmetries. Let us show a straightforward one.

\begin{proposition}[Slice symmetry]\label{prop:slicesym}
    Let $\pmb{\lambda}$ be a $d$-tuple of partitions of $m$ and $1 < i < j \le \ell(\lambda^{(1)})$ be two rows of its first partition. Fix $T \in A(\pmb{\lambda})$ and let $T'$ be the table obtained by swapping the letters $i$ and $j$ 
    in the first row, i.e. relabelling of blocks. Then the weight of $T'$ is $((i,j)\lambda^{(1)},\lambda^{(2)},\ldots)$, but the vectors
    $$
        \Delta_T = (-1)^{\lambda_i^{(1)} \cdot \lambda_j^{(1)}}\Delta_{T'}, \qquad \nabla_T = (-1)^{\lambda_i^{(1)} \cdot \lambda_j^{(1)}}\nabla_{T'},
    $$
    are still highest weight vectors of weight $\pmb{\lambda}$. 
\end{proposition}
\begin{proof}
    By definition of polynomials and forms, the numbering of the blocks only affect the signature $(-1)^T$. Since $\lambda_i^{(1)} \cdot \lambda_j^{(1)}$ swaps are performed, the signature changes by this amount.
\end{proof}

\begin{remark}
    In terms of hypermatrices, the swap of block labels means slice swaps. Thus, when we consider index set of the highest weight vectors, we may narrow it to the orbits of hypermatrices $A^\lex_d$ and $B^\lex_d$, by considering only orbits under the swaps of certain slices.
\end{remark}

\subsubsection{Horizontal concatenation}
If we have two tables $P,Q \in A_d$, it is natural to consider some operations on the tables. 

\begin{definition}[Horizontal concatenation]
    Let $\lambda \vdash m$ be a partition. For a positive integer sequence $\alpha$, which is a permutation of $\lambda$, let $w \in A(\alpha)$ be a word of weight $\alpha$. By $w + x$ for $x\in \mathbb{N}$ denote the word with each $w_i$ increased by $x$, i.e. $(w + x)_i = w_i + x$.

    For a pair of words $(p, q) \in A(\alpha)\times A(\beta)$ of weights $\alpha = (\alpha_1,\ldots,\alpha_{\ell})$ and $\beta = (\beta_1,\ldots,\beta_{r})$ (not necessarily partitions), define the new word 
    $$
        p\oplus q := p.(q + \ell) \in A(\alpha.\beta)
    $$
    where `.' denotes the concatenation of words. 
    For example, if 
    $
        p = 11122,$  
        $q = 122333,$ then 
        $p\oplus q = p.(q+2) = 11122\mathbf{344555},
    $
    where $\alpha = (3,2), \beta = (1,2,3)$ and the resulting word has weight $(3,2,1,2,3)$. If $\alpha$ and $\beta$ are partitions, then the resulting weight is $w(p \oplus q) = \alpha \uplus \beta := (\alpha' + \beta')'$, i.e. vertical sum.

    Similarly, we extend this operation to $d$-tables. Fix $d$ and let $\pmb{\lambda}$ and $\pmb{\mu}$ be $d$-tuples of partitions of $m_1$ and $m_2$, respectively. Let $P = (p_1,\ldots p_d) \in A(\pmb{\lambda})$ and $Q = (q_1,\ldots,q_d) \in A_d(\pmb{\mu})$ be $d \times m_1$ and $d \times m_2$ tables. Define $d \times (m_1 + m_2)$ table $P \oplus Q \in A(\pmb{\lambda} \uplus \pmb{\mu})$ by setting $i$-th row $(P \oplus Q)_i = p_i \oplus q_i$. For example, for $d = 2$ we have:
    $$
        P = \begin{pmatrix}
            11122\\
            11122
        \end{pmatrix}, 
        Q = \begin{pmatrix}
            111123\\
            112323
        \end{pmatrix},
        P \oplus Q = \begin{pmatrix}
            11122\textbf{333345}\\
            11122\textbf{334545}
        \end{pmatrix},
    $$
    which in terms of matrices corresponds to
    $$
    P = \begin{tabular}{ | c | c | } 
        \hline
        3 & 0 \\ 
        \hline
        0 & 2 \\
        \hline
    \end{tabular}, \quad
    Q = \begin{tabular}{ | c | c| c | } 
        \hline
        2 & 1 & 1 \\ 
        \hline
        0 & 1 & 0 \\
        \hline
        0 & 0 & 1 \\
        \hline
    \end{tabular}, \quad
    P \oplus Q = \begin{tabular}{ |c|c|c|c|c| } 
        \hline
        3 & 0 & 0 & 0 & 0\\
        \hline
        0 & 2 & 0 & 0 & 0\\
        \hline
        0 & 0 & \bf{2} & \bf{1} & \bf{1} \\ 
        \hline
        0 & 0 & \bf{0} & \bf{1} & \bf{0} \\
        \hline
        0 & 0 & \bf{0} & \bf{0} & \bf{1} \\
        \hline
    \end{tabular},
    $$
    i.e. block-diagonal concatenation of (hyper)matrices.
\end{definition}

\begin{proposition}[Ring product rule]\label{prop:ring}
    For a $d \times m_1$ table $P \in A(\pmb{\lambda'})$ and a $d \times m_2$ table $Q \in A(\pmb{\mu'})$ as above, we have 
    \begin{align*}
        \Delta_{P} \vee \Delta_Q 
            = \Delta_{P \oplus Q} 
            \in \mathrm{HWV}_{\pmb{\lambda} + \pmb{\mu}} \bigvee^{m_1 + m_2} (\mathbb{C}^n)^{\otimes d}
        \\
        \nabla_P \wedge \nabla_Q 
            = \nabla_{P \oplus Q} 
            \in \mathrm{HWV}_{\pmb{\lambda} + \pmb{\mu}} \bigwedge^{m_1 + m_2} (\mathbb{C}^n)^{\otimes d},
    \end{align*}
    where $n \ge \max(\ell(\lambda), \ell(\mu))$.
\end{proposition}
\begin{proof}
    First note that the vectors $\Delta_{P \oplus Q}$ and $\nabla_{P \oplus Q}$ are indeed highest weight vectors of weight $\pmb{\lambda} + \pmb{\mu}$ by Proposition~\ref{prop:slicesym}.
    For $(\Lambda,{\textstyle\bigDiamond}) = (\Delta, \vee)$ or $(\Lambda,{\textstyle\bigDiamond}) = (\nabla, \wedge)$, by definition we have:
    $$
        \Lambda_{P \oplus Q} = (-1)^{P\oplus Q}\sum_{R \in A(\pmb{\lambda}+\pmb{\mu})} \sgn_{P \oplus Q}(R) {\textstyle\bigDiamond}_{i=1}^{m_1 + m_2} e_{R(i)},
    $$
    Note that $$\sgn_{P\oplus Q}(R) = \sgn_P(R(1),\ldots,R(m_1)) \cdot \sgn_Q(R(m_1 + 1),\ldots,R(m_1 + m_2)),$$ 
    since the first $m_1$ columns are independent of the rest $m_2$ columns.
    Thus, substituting $R_1 := (R(1),\ldots,R(m_1))$ and $R_2 := (R(m_1+1),\ldots,R(m_1+m_2))$ and taking out the brackets we obtain:
    $$
    \Lambda_{P \oplus Q} = (-1)^{P\oplus Q} (-1)^P (-1)^Q \Lambda_P {\textstyle\bigDiamond} \Lambda_Q.
    $$
    Since, $(-1)^{P\oplus Q} = (-1)^P (-1)^Q$, the statement follows.
\end{proof}

Let us show how this implies the well-known {\it semigroup} property of Kronecker coefficients for general $d$, which was proved in \cite{chm} (for $d = 3$).
\begin{corollary}[Semigroup property]
    If $g(\pmb{\lambda}) \cdot g(\pmb{\mu}) > 0$, then $g(\pmb{\lambda} + \pmb{\mu}) \ge \max(g(\pmb{\lambda}), g(\pmb{\mu}))$.
\end{corollary}
\begin{proof}
    Let $a = g(\pmb{\lambda}) \ge g(\pmb{\mu})$. By Theorem~\ref{th:spanning-sets} there is a basis 
    $$
    \mathrm{span}\{\Delta_{P_1},\ldots, \Delta_{P_a}\}
    = \mathrm{HWV}_{\pmb{\lambda}} \bigvee^{m_1} (\mathbb{C}^{n})^{\otimes d}
    $$
    for $P_i \in B(\pmb{\lambda'})$. Let $Q \in B(\pmb{\mu'})$ be such that $\Delta_Q \neq 0$. Then
    $$
    \mathrm{span}\{\Delta_{P_1 \oplus Q},\ldots, \Delta_{P_a \oplus Q}\} 
    \subseteq 
    \mathrm{HWV}_{\pmb{\lambda} + \pmb{\mu}}
        \bigvee^{m_1+m_2} 
        (\mathbb{C}^{n})^{\otimes d}
    $$
    and the polynomials $\{\Delta_{P_i\oplus Q}\}$ are linearly independent, since the polynomial ring is an integral domain.
\end{proof}


\section{Duality of highest weight vectors}\label{sec:duality}

In the section we prove duality for highest weight polynomials and forms. 

\vspace{0.5em}


For $S \in A(\pmb{\lambda})$ denote by $\mathrm{stab}(S) := \{ \pi \in S_m: \pi S = S \}$ the stabilizer w.r.t. column permutation and by $||S|| := |\mathrm{stab}(S)|$ its size.
For a partition $\lambda = (\lambda_1, \ldots, \lambda_{\ell})$ we also define the sign $(-1)^{\lambda} := (-1)^{w}$ where $w = \underbrace{12\ldots \lambda_1}_{1\text{-st block}}\ldots \underbrace{12\ldots\lambda_{\ell}}_{\ell\text{-th block}}$, and $(-1)^{\pmb{\lambda}} := \prod_{i =1}^{d} (-1)^{\lambda^{(i)}}$ for $\pmb{\lambda} = (\lambda^{(1)}, \ldots, \lambda^{(d)})$. 

\begin{theorem}
\label{th:duality2}
    Let $\pmb{\lambda}$ be a $d$-tuple of partitions of $m$. 
    
    (i) For odd $d$ we have 
    \begin{align*}
        \Delta_{T} &= (-1)^{\pmb{\lambda}} \sum_{S \in A^\lex(\pmb{\lambda})} \langle \nabla_S, \bigwedge e_{T} \rangle \bigvee e_S, \quad T \in B(\pmb{\lambda'}), \\
        \nabla_{S} &= (-1)^{\pmb{\lambda}} \sum_{T \in B^\lex(\pmb{\lambda'})} \langle \Delta_T, \bigvee e_{S} \rangle \bigwedge e_T, \quad S \in A(\pmb{\lambda}),
    \end{align*}
    
    (ii) For even $d$ we have 
    \begin{align*}
        \Delta_{S} &= (-1)^{\pmb{\lambda}} \sum_{S' \in A(\pmb{\lambda'})} 
            \langle \Delta_{S'}, \bigvee e_S \rangle \bigvee e_{S'}, \quad S \in A(\pmb{\lambda}), \\
        \nabla_{T} &= (-1)^{\pmb{\lambda}} \sum_{T' \in B(\pmb{\lambda'})} 
            \langle \nabla_{T'}, \bigwedge e_T \rangle \bigwedge e_{T'}, \quad T \in B(\pmb{\lambda}).
    \end{align*}    
\end{theorem}

To prove this, we first introduce the coefficients $a(T, S)$ and $b(S, T)$.

\begin{definition}
    Let $(S,T) \in A(\pmb{\lambda}) \times A(\pmb{\lambda'})$. After symmetrization (or alternation) of $P_T$ (or $P_S$) (see Section~\ref{sec:spanningset}) we obtain $\Delta_T$ (or $\nabla_S$) with some collinear terms in the sum. Let us collect them to obtain
    \begin{align*}
        \Delta_T = \frac{1}{||T||}\sum_{S \in A(\pmb{\lambda})}(-1)^T\sgn_T(S) \bigvee e_{S}
        =& \sum_{S \in A^{\mathrm{lex}}(\pmb{\lambda})} a(T,S) \bigvee e_S,
        \\
        \qquad\text{where } a(T,S) :=& \frac{(-1)^{T}}{||T||\cdot||S||}\sum_{\pi \in S_m} \sgn_T(\pi S),
        \\
        \nabla_S = \frac{1}{||S||}
        \sum_{T \in B(\pmb{\lambda'})}
            (-1)^S \sgn_S(T) \bigwedge e_{T}
        =& \sum_{T \in B^{\mathrm{lex}}(\pmb{\lambda'})} b(S,T) \bigwedge e_T,
        \\
        \text{where } b(S,T) :=& \frac{(-1)^S}{||S||}  \sum_{\pi \in S_m} \sgn(\pi) \cdot \sgn_S(\pi T).
    \end{align*}
    Then $\langle\Delta_T, \bigvee e_S\rangle = a(T,S)\in \mathbb{Z}$ and $\langle\nabla_S, \bigwedge e_T\rangle = b(S,T) \in \mathbb{Z}$.
\end{definition}

\begin{remark}
    For $\nabla_S$ it is straightforward, since columns of $T$ do not repeat. For $\Delta_T$, each term in the sum of $a(T,S)$ repeats $||S||=|\mathrm{Stab}(S)|$ times, which then cancels out.
\end{remark}

It is clear that the sums in the definition of $a(T,S)$ or $b(S,T)$ may contain vanishing terms. We are going to describe these coefficients with nonzero terms. 

\begin{definition}
    For $d \times m$ tables $X, Y,$ 
    define the following set of permutations
    $$
        G(X,Y) := \{ w \in S_m: \sgn_X(w Y) \neq 0 \}. 
    $$
    Using this set we can rewrite the coefficients $a(T,S)$ and $b(S,T)$ as follows:
    \begin{equation}\label{eq:ats-bst}
        a(T,S) = \frac{1}{||T||\cdot||S||} \sum_{\pi \in G(T,S)} \sgn'_T(\pi S),\quad
        b(S,T) = \frac{1}{||S||}\sum_{\pi \in G(S,T)} \sgn(\pi) \cdot \sgn'_S(\pi T).
    \end{equation}
    where we write:
    $$
        \sgn'_X(\pi Y) := (-1)^X \sgn_X(\pi Y)
    $$
    to shorten the notation. 
    \end{definition}

    Let $\alpha = (\alpha_1,\ldots,\alpha_\ell) \in \mathbb{N}^\ell$ be an integer sequence with $|\alpha| = m$. For $t \in A(\alpha)$, let $Y_{t} \subseteq S_m$ be the stabilizer of the word $t$, known as the \textit{Young subgroup}.

\subsection{Single row}
We first show properties of $G(s,t)$ for single row tables $s, t$. We write lowercase letters for single row tables. 

\begin{lemma}\label{lemma:gst}
    Let $\lambda\vdash m$ be a partition of length $\ell = \ell(\lambda)$. For any pair of words $(s,t) \in A(\lambda) \times A(\lambda')$ the following properties hold:
    \begin{enumerate}[label=(\alph*)]
        \item  $G(s,t) \in Y_{s} \backslash S_m / Y_{t}$ is a double coset. Moreover, for $h \in G(s,t)$ and double coset decomposition $w = xhy \in Y_s h Y_t$ we have $\sgn_s(wt) = \sgn(x)\cdot \sgn_s(ht)$.
        \item $\sgn_s(t) \neq 0 \iff \sgn_t(s) \neq 0 \iff \{(s_i, t_i)\}_{i=1,\ldots,m} = \{(i,j) \in D(\lambda) \}$ is the Young diagram of $\lambda$, where $s = s_1\ldots s_m$ and $t = t_1 \ldots t_m$. 
        \item $\sgn_{s}(t) \neq 0 \iff \sgn_{hs}(ht) \neq 0$ for any $h\in S_m$.
        \item $G(s,t)^{-1} = G(t,s)$.
        \item The product $(-1)^{G(s,t)} := \sgn_s(\pi t) \cdot (-1)^{\pi} \cdot \sgn_{t}(\pi^{-1}s)$ does not depend on $\pi \in G(s,t)$, i.e. it is invariant on the double coset.
        \item The product $f(s,t) := (-1)^s(-1)^{G(s,t)}(-1)^t$
        does not depend on a pair $(s,t)$, but only on $\lambda$ and is equal to $(-1)^{\lambda} := (-1)^{t_0}$ for $t_0 = \underbrace{12\ldots \lambda_1}_{1\text{-st block}}\ldots \underbrace{12\ldots\lambda_{\ell}}_{\ell\text{-th block}}$.
    \end{enumerate}
\end{lemma}

\begin{proof}
    We prove the statements one by one.
    
    (a) Due to the weight of the words $s$ and $t$ we have $G(s,t) \neq \varnothing$. Let us fix $h \in G(s,t)$, i.e. $\sgn_s(ht) \neq 0$.
	For an arbitrary permutation $w \in G(s,t)$ with $\sgn_{s}(w t) \neq 0$, there is a unique permutation $x \in Y_{s}$ which permutes elements only in blocks of $s$ so that the words $x\cdot h t = w \cdot t$ are equal. For $y := (xh)^{-1} w$ we have $y \cdot t = t$, and hence $y \in Y_t$. Therefore, the decomposition $w = xhy \in Y_s h Y_t$ holds. 
	
	Conversely, suppose that $w = xhy \in Y_s h Y_t$. Then $y$ leaves $t$ unchanged, $h$ makes $\sgn_s(ht) \neq 0$. Since each transposition of $x$ to the word $hy \cdot t$ changes the sign by $-1$, we have $\sgn_s(wt) = \sgn_s(xht) = (-1)^{x} \cdot \sgn_s(ht) \neq 0$ and $w \in G(s,t)$.
	
	\vspace{0.5em}

    (b) Let us present $s$ and $t$ as a $2 \times m$ table with the first row $s$ and the second row $t$, which we denote by 
    $$(s,t) := \begin{pmatrix}
        s_1 & \ldots & s_m\\
        t_1 & \ldots & t_m
    \end{pmatrix}.
    $$ 
    Then $\sgn_s(t) \neq 0$ if and only if the set of columns of the table $(s,t)$ is equal to the set of cell coordinates of the diagram $\lambda$. 
    
    \vspace{0.5em}

    (c) Swapping rows or columns does not violate this property. 
    
    \vspace{0.5em}

    (d) Let $w^{-1} \in G(s,t)^{-1}$, i.e. $\sgn_s(w^{-1} t) \neq 0$. Then by (b) we have:
    $$
        \sgn_s(w^{-1}t) \neq 0 \iff \sgn_{w^{-1}t}(s) \neq 0 \iff \sgn_{t}(ws) \neq 0,
    $$
    i.e. $w \in G(t,s)$ and all steps are reversible.
    
    \vspace{0.5em}

    (e) By (a) we can write $G(s,t) = Y_s h Y_t$ for an arbitrary element $h \in G(s,t)$. Since Young subgroups are generated by transpositions, this double coset can also be generated with multiplying $h$ by transpositions of $Y_s$ from the left and by transpositions of $Y_t$ from the right. 
    For any transposition $\alpha \in Y_s$, replacing $h \to \alpha h$ we have: $\sgn_s(\alpha h t) = -\sgn(h t)$ (transposition within $s$-block), $\sgn_t((\alpha h)^{-1}s) = \sgn_t(h^{-1} s)$ (transposition of equal letters) and $(-1)^{\alpha h} = -(-1)^{h}$, thus the product is unchanged. 
    Similarly, for a transposition $\beta \in Y_t$ replacing $h \to h \beta$ we have: $\sgn_s(w\beta t) = \sgn_s(wt)$, $\sgn_t((h\beta)^{-1} s) = -\sgn_t(hs)$ and $(-1)^{h\beta} = -(-1)^{h}$. Thus, the product is indeed invariant on the double coset.
    
    \vspace{0.5em}

    (e) Let $s_0 = 1^{\lambda_1}\ldots \ell^{\lambda_\ell}$ and $t_0 = 12\ldots\lambda_1\ldots12\ldots\lambda_\ell$. 
    We will show consequently that $f(s,t) = f(s_0, t_0)$. We will do that in two steps: first, we replace $t \to ht$ so that $\sgn_s(ht) \neq 0$, i.e find a nonzero shift; then we sort the columns of $(s,ht)$ to achieve $(s_0, t_0)$. Diagrammatically we have:
    $$
        \begin{pmatrix}
            s\\
            t
        \end{pmatrix} 
        \xlongrightarrow{t \to ht}
        \begin{pmatrix}
            s\\
            ht
        \end{pmatrix} 
        \xlongrightarrow{\text{ sort columns }}
        \begin{pmatrix}
            s_0\\
            t_0
        \end{pmatrix} = 
        \begin{pmatrix}
            1^{\lambda_1}& \ldots & \ell^{\lambda_\ell}\\
            12\ldots\lambda_1&\ldots& 12\ldots\lambda_\ell
        \end{pmatrix}
    $$
    
    Let $h \in G(s,t)$ be the permutation of minimal length satisfying $\sgn_s(ht) \neq 0$. This permutation does not permute equal letters of $t$ (otherwise the length can be reduced). 
    Let us show that $f(s,t) = f(s, ht)$. By (c) we have:
    \begin{align*}
        f(s,t) &= (-1)^s\sgn_s(ht) \cdot (-1)^h \cdot \sgn_t(h^{-1} s) (-1)^t,\\ 
        f(s,ht) &= (-1)^s \sgn_s(ht) \cdot (-1)^{\mathrm{id}} \cdot \sgn_{ht}(s) (-1)^{ht}.
    \end{align*}
    Note that $(-1)^{ht} = (-1)^h (-1)^t$, since each transposition of $h$ swaps only distinct letters of $t$. It remains to show that $\sgn_t(h^{-1} s) = \sgn_{ht}(s)$. Indeed, the table $(t, h^{-1}s)$ is obtained from the table $(ht, s)$ by permuting its columns according to the permutation $h$, which does not permute equal letters of $t$. Hence, the relative order 
    of $t$-block permutations in the bottom row does not change in both tables; this implies that $\sgn_t(h^{-1} s) = \sgn_{ht}(s)$ and $f(s,t) = f(s,ht)$.

    Now we may assume that $\sgn_s(t) \neq 0$ (i.e. replace $ht \to t$), then the table $(s,t)$ contains distinct columns forming  the set $\{(i,j)\in D(\lambda)\}$. 
    By (b) we can reach the table $(s_0,t_0)$ from the table $(s,t)$ by column swaps. Let us consequently apply adjacent transpositions of the form $(i,i+1)$ to the columns of $(s,t)$ to obtain $(s_0, t_0)$ and observe that $f(s,t)$ does not change. By (d), put $\pi = e$ to rewrite
    $$
        f(s,t) = (-1)^s \sgn_s(t) \cdot \sgn_t(s) (-1)^t.
    $$
    Let $\alpha=(i,i+1)$ be the transposition we apply. Let us keep track of the sign changes after applying $\alpha$. If $i$-th and $(i+1)$-th columns are: 
    \begin{itemize}
        \item $\begin{bmatrix}
            xx\\
            yz
        \end{bmatrix}$ then the changes are $\sgn_{\alpha s}(\alpha t) \to -\sgn_s(t)$ and $(-1)^{\alpha t} = -(-1)^t$;
        \item $\begin{bmatrix}
            yz\\
            xx
        \end{bmatrix}$ then the changes are $\sgn_{\alpha t}(\alpha s) \to -\sgn_s(t)$ and $(-1)^{\alpha s} = -(-1)^s$;
        \item $\begin{bmatrix}
            xy\\
            zw
        \end{bmatrix}$ then the changes are $(-1)^{\alpha s} \to -(-1)^s$ and $(-1)^{\alpha t} \to -(-1)^{t}$,
    \end{itemize}
    where $x,y,z,w$ are distinct letters.
    The case where columns are equal is impossible. As we can see in each case $f(s,t)$ does not change, hence the statement follows.
\end{proof}
\begin{remark}
    Above discussions automatically imply $(-1)^{\pmb{\lambda}} = (-1)^{\pmb{\lambda'}}$.
\end{remark}

\subsection{Multiple rows}
Let us now show properties of sets $G(S,T)$ for $d > 1$.

\begin{lemma}\label{lemma:ats}
	Let $(S, T) \in A(\pmb{\lambda}) \times A(\pmb{\lambda'})$ with $S = (s_1,\ldots,s_d)$ and $t = (t_1,\ldots,t_d)$. 
	The following properties hold:
	\begin{itemize}
		\item[(a)] $G(S,T) = \bigcap_{i=1}^d Y_{s_i} \cdot h_i \cdot Y_{t_i}$  for elements $h_i \in G(s_i, t_i)$.
		\item[(b)] $G(S,T) = G(T,S)^{-1}$.
		\item[(c)] For odd $d$ we have $a(T,S) = (-1)^{\pmb{\lambda}}\, b(S,T)$.
		\item[(d)] For even $d$ we have $a(T,S) = (-1)^{\pmb{\lambda}}\, a(S,T)$ and $b(S,T) = (-1)^{\pmb{\lambda}}\, b(T,S)$.
	\end{itemize}	 
\end{lemma}

\begin{proof} 
	(a) A permutation $w \in S_m$ satisfies $w \in G(S,T)$ if and only if $\sgn_S(wT) \neq 0$, which is by definition $\sgn_S(wT) = \prod_{i=1}^d\sgn_{s_i}(wt_i) \neq 0$. Hence, $w \in G(s_i, t_i)$ or equivalently, 
	$$
		G(S,T) = \bigcap_{i=1}^d G(s_i, t_i) = \bigcap_{i=1}^d Y_{s_i}h_iY_{t_i}.
	$$
	This implies that  by Lemma~\ref{lemma:gst} for each $w \in G(S,T)$ and row $i$ we have the decomposition $w = x_i h_i y_i$ with $x_i \in Y_{s_i}$, $y_i \in Y_{t_i}$ and $h_i \in G(s_i,t_i)$. Note that $x_i$ and $y_i$ depend on $w$. 

    \vspace{0.5em}

	(b) For $w \in G(S,T)$ and each $i \in [d]$, if $w = x_i h_i y_i$ then $w^{-1} = (y_i)^{-1} h^{-1}_i (x_i)^{-1} \in Y_{t_i} h^{-1}_i Y_{s_i}$, and hence $w^{-1} \in G(T,S)$. Since the group inverse operation is bijective we have $G(S,T)^{-1} = G(T,S)$. 

    By Lemma~\ref{lemma:gst}(e) for $w \in G(S,T)$ we have $\sgn'_{s_i}(w t_i) \cdot \sgn'_{t_i}(w^{-1} s_i) = (-1)^{w}(-1)^{\lambda^{(i)}}$. This implies the following correspondence for the terms of $b(S,T)$ and $a(T,S)$:
    \begin{align}\label{eq:sign-identity}
        \sgn'_S(wT) 
        &= \prod_{i=1}^d \sgn'_{s_i}(w t_i) \nonumber 
        = \prod_{i=1}^d (-1)^{\lambda^{(i)}} (-1)^{w}\cdot \sgn'_{t_i}(w^{-1} s_i) \nonumber 
        \\
        &= (-1)^{\pmb{\lambda}}\, ((-1)^{w})^d \cdot \sgn'_T(w^{-1}S). 
    \end{align}

	(c) For odd $d$ by (b) and the identity \eqref{eq:sign-identity} we obtain
	\begin{align*}
		||S|| \cdot b(S,T) &= \sum_{w \in G(S,T)}(-1)^w\sgn'_S(wT)
		= (-1)^{\pmb{\lambda}}\sum_{w^{-1} \in G(T,S)} \sgn'_T(w^{-1}S) \\
		&= (-1)^{\pmb{\lambda}}\, a(T,S) \cdot ||S||,
	\end{align*}
    where $||T|| = 1$, otherwise the identity is trivial. Thus, $b(S,T) = (-1)^{\pmb{\lambda}} a(T,S)$.

    (d) For even $d$ we have $\sgn'_S(wT) = (-1)^{\pmb{\lambda}} \sgn'_T(w^{-1} S)$ in \eqref{eq:sign-identity} and
    \begin{align*}
        ||T|| \cdot a(T,S) \cdot||S||
        &= \sum_{w \in G(T,S)} \sgn'_{T}(w S)
        = (-1)^{\pmb{\lambda}} \sum_{w^{-1}  \in G(S,T)} \sgn'_{S}(w^{-1}  T)\\
        &= (-1)^{\pmb{\lambda}}\, ||S|| \cdot a(S, T) \cdot ||T||, \\
        ||S|| \cdot b(S, T)
        &= \sum_{w \in G(S,T)} (-1)^{w} \cdot \sgn'_{S}(w T)
        = (-1)^{\pmb{\lambda}}\sum_{w^{-1}  \in G(T,S)} (-1)^{w^{-1}} \cdot \sgn'_{T}(w^{-1}  S) \\ 
        &= (-1)^{\pmb{\lambda}}\, ||T|| \cdot b(T,S).
    \end{align*}
    Note, that if $b(S,T) \neq 0$ then $S,T \in B_d$, therefore $||S||=||T|| = 1$.
\end{proof}

\begin{proof}[Proof of Theorem~\ref{th:duality2}]
(i) If $d$ is odd, then by Lemma~\ref{lemma:ats}(c) we have:
\begin{align*}
\langle \Delta_T, \bigvee e_S \rangle = a(T,S) = (-1)^{\pmb{\lambda}}\, b(S,T) = (-1)^{\pmb{\lambda}} \langle\nabla_{S}, \bigwedge e_T\rangle.
\end{align*}
Hence,
\begin{align*}
\Delta_T &= \sum_{S \in A(\pmb{\lambda})} a(T,S) \bigvee e_S =  (-1)^{\pmb{\lambda}} \sum_{S \in A(\pmb{\lambda})} \langle\nabla_{S}, \bigwedge e_T\rangle \bigvee e_S,
\\
\nabla_S &= \sum_{T \in B(\pmb{\lambda'})} b(S,T) \bigwedge e_T = (-1)^{\pmb{\lambda}} \sum_{T \in B(\pmb{\lambda'})}   \langle\Delta_{T}, \bigvee e_S\rangle \bigwedge e_T.
\end{align*}

(ii) If $d$ is even, then by Lemma~\ref{lemma:ats}(d) we have: 
$$
a(S,S') = (-1)^{\pmb{\lambda}}\, a(S',S) =(-1)^{\pmb{\lambda}} \langle 
        \Delta_{S'}, \bigvee e_{S}
    \rangle, 
$$
$$ 
    b(T,T') = (-1)^{\pmb{\lambda}}\, b(T',T) = (-1)^{\pmb{\lambda}}\langle 
        \Delta_{T'}, \bigwedge e_{T}
    \rangle.
$$ Hence, 
\begin{align*}
    \Delta_{S} &=  (-1)^{\pmb{\lambda}}\sum_{S' \in A(\pmb{\lambda})} a(S,S')\bigvee e_{S'}
    =  (-1)^{\pmb{\lambda}}\sum_{S' \in A(\pmb{\lambda})}
        \langle
            \Delta_{S'}, 
            \bigvee e_{S}\rangle
        \bigvee e_{S'},
        \\
    \nabla_{T} &=  (-1)^{\pmb{\lambda}}\sum_{T' \in  B(\pmb{\lambda})} b(T,T') \bigwedge e_{T'}
    =  (-1)^{\pmb{\lambda}} \sum_{T' \in  B(\pmb{\lambda})}
        \langle
            \nabla_{T'}, 
            \bigwedge e_T\rangle 
        \bigwedge e_{T'}. 
\end{align*}
This completes the proof.
\end{proof}

\subsection{Interpretation of  coefficients}
Let us now give a combinatorial interpretation of the coefficients $a(T,S)$ (and $b(S, T)$), for $T\in A^{\mathrm{lex}}(m)$ and $S \in A^{\mathrm{lex}}(m)$.

Let $\sigma$ be a bijective map $\sigma:T \to S$ (where $T, S$ are viewed as \textit{multisets} of columns), which is read as follows. 
Consider $T \subseteq [k]^d$ as a choice of some cells of a cube, which are going to be filled with columns from $S$. Each cell (column) $p$ of $T$ is filled by certain entry (column) $q$ of $S$, namely if $\sigma(p) = q$. This way, each cell $p \in [k]^d$ has $T_p$ entries and entry $q \in S$ is present $S_q$ times. The entries within same cell are ordered.

We call bijection $\sigma:T \to S$ a {\it $(T, S)$-filling} if additionally $\sgn_T(\sigma(T)) \neq 0$, i.e. for each direction $i \in [d]$ and for each $j$-slice $E^{(i)}_j$ in $i$-th direction: the entries of $\sigma$ in slice $E^{(i)}_j$ must have a permutation if all coordinates except the $i$-th one are ignored. More formally, if $T \cap E = \{q^1 < \ldots < q^\ell \}$ and $\sigma(q^{k}) = p^{k} \in [n]^d$, then we must have $(p^{1}_i,\ldots,p^{\ell}_i) \in S_\ell$.

The {\it sign} $\sgn(\sigma)$ of $(T,S)$-filling $\sigma$ is then the product of signs $\sgn(p^1_i \ldots p^\ell_i)$ over all slices $E^{(i)}_j$ for $i = 1,\ldots, d$ and $j = 1,\ldots, k$.
This describes the following interpretation.
\begin{proposition}\label{prop:interpretation}
    For $(T, S) \in A(\pmb{\lambda'}) \times A(\pmb{\lambda})$ we have 
    $$
        a(T,S) = (-1)^{T} \sum_{\sigma: T \to S} \sgn(\sigma)
    $$
    where $\sigma$ runs over all $(T,S)$-fillings. 
\end{proposition}
\begin{proof}
We have
$$
    a(T,S) = \frac{(-1)^T}{||S||\cdot ||T||} \sum_{\sigma \in S_m} \sgn_T(\sigma S) = (-1)^T\sum_{\sigma: T \to S\text{ - injective}} \sgn_T(\sigma(T)),
$$
where we match $\sgn(p^1_i \ldots p^\ell_i)$ with the sign of the corresponding block in $i$-th row of $T$ (each slice $E$ corresponds to certain row-block). The factor $||T||$ ensures that the order within the cells does not matter.
\end{proof}

\section{Duality isomorphism of highest weight spaces}\label{sec:isomorphism}
Let $V = (\mathbb{C}^n)^{\otimes d}$, $\pmb{\lambda}$ be a $d$-tuple of partitions and recall that we have: 
$$
\begin{matrix}
    \text{odd }d:&
    &\mathrm{HWV}_{\pmb{\lambda}} \bigvee V 
    = \mathrm{span}\{ \Delta_T: T\in B^{\mathrm{lex}}(\pmb{\lambda'}) \},
    &\mathrm{WV}_{\pmb{\lambda}} \bigvee V 
    = \langle \bigvee e_S: S \in A^{\mathrm{lex}}(\pmb{\lambda})\rangle,
    \vspace{0.5em}
    \\
    &~& \mathrm{HWV}_{\pmb{\lambda'}} \bigwedge V 
    = \mathrm{span}\{ \nabla_S: S \in A^{\mathrm{lex}}(\pmb{\lambda}) \},
    &\mathrm{WV}_{\pmb{\lambda'}} \bigwedge V 
    = \langle \bigwedge e_T: T \in B^{\mathrm{lex}}(\pmb{\lambda'}) \rangle;
    \vspace{1em}
    \\
    \text{even }d:&
    &\mathrm{HWV}_{\pmb{\lambda}} \bigvee V 
    = \mathrm{span}\{ \Delta_S: S\in A^{\mathrm{lex}}(\pmb{\lambda'}) \},
    &\mathrm{WV}_{\pmb{\lambda}} \bigvee V 
    = \langle \bigvee e_S: S \in A^{\mathrm{lex}}(\pmb{\lambda})\rangle,
    \vspace{0.5em}
    \\
    &~&\mathrm{HWV}_{\pmb{\lambda'}} \bigwedge V 
    = \mathrm{span}\{ \nabla_T: T \in B^{\mathrm{lex}}(\pmb{\lambda}) \},
    &\mathrm{WV}_{\pmb{\lambda'}} \bigwedge V 
    = \langle \bigwedge e_T: T \in B^{\mathrm{lex}}(\pmb{\lambda'}) \rangle,
\end{matrix}
$$
where $\mathrm{supp}(T) \subseteq [n]^d$ and $\mathrm{supp}(S) \subseteq [n]^d$. 

\begin{definition}(Index maps)
    Define the linear maps: 
    \begin{align*}
        \text{for odd }d:\ \ 
        \Delta: \bigwedge V \longrightarrow \mathrm{HWV} \bigvee V, \quad
        \nabla: \bigvee V \longrightarrow \mathrm{HWV} \bigwedge V,\\
        \text{for even }d:\ \ 
        \Delta: \bigvee V \longrightarrow \mathrm{HWV} \bigvee V, \quad
        \nabla: \bigwedge V \longrightarrow \mathrm{HWV} \bigwedge V,
    \end{align*}
   on the standard bases as follows  
    $$
    \begin{matrix}
        \text{for odd }d:& \qquad
        &\Delta: \mathrm{WV}_{\pmb{\lambda}} \bigwedge V 
        \longrightarrow \mathrm{HWV}_{\pmb{\lambda'}} \bigvee  V,
        \qquad\bigwedge e_T \longmapsto \Delta_T,\qquad\qquad
        \vspace{0.5em}
        \\
        &~
        &\nabla: \mathrm{WV}_{\pmb{\lambda}} \bigvee V 
        \longrightarrow \mathrm{HWV}_{\pmb{\lambda'}} \bigwedge V, 
        \qquad\bigvee e_S \longmapsto \nabla_S; \qquad\qquad
        \vspace{1em}
        \\
        \text{for even }d:&
        &\Delta: \mathrm{WV}_{\pmb{\lambda}} \bigvee V 
        \longrightarrow \mathrm{WV}_{\pmb{\lambda'}} \bigvee V, 
        \qquad\bigvee e_S \longmapsto \Delta_S, \qquad\qquad
        \vspace{0.5em}
        \\
        &~
        &\nabla: \mathrm{WV}_{\pmb{\lambda}} \bigwedge V 
        \longrightarrow \mathrm{WV}_{\pmb{\lambda'}} \bigwedge V, 
        \qquad\bigwedge e_T \longmapsto \nabla_T.\qquad\qquad
    \end{matrix}
    $$
    and extended linearly, where $S \in A^\lex(\pmb{\lambda})$, $T \in B^\lex(\pmb{\lambda})$. The maps are well-defined due to  Lemma~\ref{lemma:colswap} on column swaps. To define the map on total wedge (or symmetric) space we allow $\pmb{\lambda}$ to be $d$-tuple of compositions as well. The weight spaces equivalent up to the slice swaps (in the sense of Proposition~\ref{prop:slicesym}) have identical images. Therefore, it is enough to restrict $\pmb{\lambda}$ to partitions. 
\end{definition}

\begin{remark}\label{remark:domain}
The domain $\bigwedge V$ of the map $\Delta$ (resp. $\nabla$) can be also restricted to the Weyl group orbits $\bigwedge V/W$, where $W = (S_n)^d$ is the group generated by slice transpositions, the Weyl group for $\mathrm{GL}(n)^d$. 
That is, for transposition $(i,j) \in S_n$ and $\wedge e_T \in \bigwedge V$ with $T \in B^{\mathrm{lex}}(\lambda,\ldots)$ the action is defined by $(i,j) \cdot \wedge e_T := (-1)^{\lambda_i \lambda_j}\wedge e_{T'}$, where table $T'$ is of weight $(\ldots,\lambda_{i+1},\lambda_i,\ldots)$ and results in switching symbols $i$ and $j$ of the first row of $T$ (similarly for other rows). Then by Proposition~\ref{prop:slicesym}, for any $w \in W$ the image $\Delta(w \cdot \wedge e_T) = \Delta_T$.
See also Remark~\ref{remark:kernel}.
\end{remark}

\begin{theorem}[Duality isomorphism of highest weight spaces]\label{th:isomorphism}
    Let $\pmb{\lambda}$ be a $d$-tuple of partitions of $m$. Then the maps $\Delta$ and $\nabla$ are surjective. 

(i) For odd $d$, 
$$
\Delta: \mathrm{HWV}_{\pmb{\lambda'}} \bigwedge V 
        \longrightarrow \mathrm{HWV}_{\pmb{\lambda}} \bigvee  V,  
        \qquad 
\nabla: \mathrm{HWV}_{\pmb{\lambda'}} \bigvee V 
        \longrightarrow \mathrm{HWV}_{\pmb{\lambda}} \bigwedge V,
$$
are isomorphisms of highest weight spaces. 

(ii) For even $d$, 
$$
\Delta: \mathrm{HWV}_{\pmb{\lambda'}} \bigvee V 
        \longrightarrow \mathrm{HWV}_{\pmb{\lambda}} \bigvee  V,  
        \qquad 
\nabla: \mathrm{HWV}_{\pmb{\lambda'}} \bigwedge V 
        \longrightarrow \mathrm{HWV}_{\pmb{\lambda}} \bigwedge V,
$$
are isomorphisms of highest weight spaces. 
\end{theorem}

\begin{proof}
    Surjectivity of the maps follows from Theorem~\ref{th:spanning-sets}.
    We present the proof for the map $\Delta$ and odd $d$, 
    the other cases are very similar. 
    Since $\mathrm{HWV}_{\pmb{\lambda}} \bigvee  V \subset  \mathrm{WV}_{\pmb{\lambda}} \bigvee  V,$ 
    we can view the map $\Delta: \mathrm{WV}_{\pmb{\lambda'}} \bigwedge  V \to \mathrm{WV}_{\pmb{\lambda}} \bigvee  V$ with extended range and record its coordinates in the standard basis. For $T \in B^\mathrm{lex}(\pmb{\lambda'})$ 
    we have: 
    $$
        \Delta(\bigwedge e_T) = \Delta_T = \sum_{S \in A^\mathrm{lex}(\pmb{\lambda})} a(T,S) \bigvee e_S, 
    $$
    and hence the $|A^{\mathrm{lex}}(\pmb{\lambda})|\times |B^{\mathrm{lex}}(\pmb{\lambda'})|$ matrix $\mathbf{A}^t(S,T) = (a(T, S))^{t}$ is a transition matrix of the extended map $\Delta$. By the duality Theorem~\ref{th:duality2}, the row and column spaces of this matrix are:
    $$
        \mathrm{row}(\mathbf{A}^t) = \mathrm{HWV}_{\pmb{\lambda'}} \bigwedge  V, 
        \qquad \mathrm{col}(\mathbf{A}^t) = \mathrm{HWV}_{\pmb{\lambda}} \bigvee  V. 
    $$
    Since the restriction $\mathbf{A}^t: \mathrm{row}(\mathbf{A}^t) \to \mathrm{col}(\mathbf{A}^t)$ is an isomorphism, so is $\Delta: \mathrm{HWV}_{\pmb{\lambda'}} \bigwedge  V 
    \to \mathrm{HWV}_{\pmb{\lambda}} \bigvee  V 
    $ an isomorphism as well.
\end{proof}

\begin{remark}
In particular, the isomorphisms imply the following known equalities of the corresponding dimensions given by Kronecker coefficients: 
$$
\begin{matrix}
    \text{for odd }d:
    & \dim \mathrm{HWV}_{\pmb{\lambda}} \bigvee  V 
    = g(1 \times m, \pmb{\lambda}) = g(m \times 1, \pmb{\lambda'}) = \dim \mathrm{HWV}_{\pmb{\lambda'}} \bigwedge V, 
    \vspace{0.5em}
    \\
    \text{for even }d:
    & \dim  \mathrm{HWV}_{\pmb{\lambda}} \bigvee  V 
    = g(1 \times m, \pmb{\lambda}) = g(1 \times m, \pmb{\lambda'}) = \dim  \mathrm{HWV}_{\pmb{\lambda'}} \bigvee  V, 
    \vspace{0.5em}
    \\
    ~&\dim  \mathrm{HWV}_{\pmb{\lambda}} \bigwedge  V 
    = g(m \times 1, \pmb{\lambda}) = g(m \times 1, \pmb{\lambda'}) = \dim  \mathrm{HWV}_{\pmb{\lambda}} \bigwedge  V. 
\end{matrix}
$$
\end{remark}

\section{Algebraic relations}\label{sec:linear}
To describe our main result on algebraic relations among highest weight vectors $\Delta$ and $\nabla$ we first need to introduce the notion of boundary tables. 
\begin{definition}[Boundary tables]
    For a $d$-tuple of partitions $\pmb{\lambda}$ and 
    $S \in A(\pmb{\lambda})$ let us define the set $\partial^{(\ell)}_{i j}(S)$ of \textit{boundary tables} of $S$ as 
    the set of tables $S'$ which result in \textit{replacing a single letter $j$ in $\ell$-th row of $S$ with the letter $i$}, so that the order of columns remains unchanged. For example, for $d = \ell = 1$ and $(i, j) = (1,2)$ we have 
    $\partial^{(1)}_{1 2}(11212) = \{ (11\mathbf{1}12), (1121\mathbf{1}) \}$.
    Let us note that the set $\partial^{(\ell)}_{i j}(S)$ contains the tables of the same weight $\alpha^{(\ell)}_{i j} \pmb{\lambda}$ which changes the $\ell$-th partition $\lambda$ in  $\pmb{\lambda}$ with the partition $\lambda + e_i - e_j$. 
\end{definition}

For $\mathcal{X} \subseteq A(\pmb{\lambda})$ we use the following notation to denote the specific linear combinations of  standard basis vectors: 
$$
    \bigvee \mathcal{X} 
    := \sum_{X \in \mathcal{X}} \bigvee e_X, \quad 
    \bigwedge\mathcal{X} 
    := \sum_{X \in \mathcal{X}} \bigwedge e_X.
$$

We now state the main description of algebraic relations among the vectors $\Delta$ and $\nabla$. 

\begin{theorem}[Relations on spanning sets]\label{th:relations1}
    Let $\pmb{\lambda}$ be a $d$-tuple of partitions of $m$. 
    
    (i) For odd $d$ the following linear relations hold for $\ell \in [d]$ and $1 \le i < j \le \ell((\lambda^{(\ell)})')$: 
    \begin{align}
        & \Delta\left(\bigwedge \partial^{(\ell)}_{j i}(X)\right) 
        = \sum_{T \in \partial^{(\ell)}_{j i}(X)} \Delta_{T}
        = 0, \quad X \in B(\alpha^{(\ell)}_{i j} \pmb{\lambda'}),
        \label{eq:relations-delta1} \\
        &  
        \nabla\left(\bigvee \partial^{(\ell)}_{j i}(Y)\right)
        = \sum_{S \in \partial^{(\ell)}_{j i}(Y)}\ \nabla_{S} 
         = 0, \quad Y \in A(\alpha^{(\ell)}_{i j} \pmb{\lambda'}).
         \label{eq:relations-nabla1}
    \end{align}
    Moreover, we have 
    \begin{align*}
    \mathrm{Ker} \left(\Delta: \mathrm{WV}_{\pmb{\lambda'}} \bigwedge V \to \mathrm{HWV}_{\pmb{\lambda}} \bigvee  V \right) 
    &= \mathrm{span} \left\{\bigwedge \partial^{(\ell)}_{j i}(X) : \ell \in [d], i < j, X \in B(\alpha^{(\ell)}_{i j} \pmb{\lambda'}) \right\},  \\
    \mathrm{Ker}\left(\nabla: \mathrm{WV}_{\pmb{\lambda'}} \bigvee V \to \mathrm{HWV}_{\pmb{\lambda}} \bigwedge  V \right) 
    &= \mathrm{span} \left\{\bigvee \partial^{(\ell)}_{j i}(Y) : \ell \in [d], i < j, Y \in A(\alpha^{(\ell)}_{i j} \pmb{\lambda'}) \right\}, 
    \end{align*}
    i.e. any linear relation among $\{\Delta_T \}$ or $\{ \nabla_S \}$ follows linearly from \eqref{eq:relations-delta1} or \eqref{eq:relations-nabla1}.
    
        (ii) For even $d$ the following linear relations hold for $\ell \in [d]$ and $1 \le i < j \le \ell((\lambda^{(\ell)})')$: 
    \begin{align}
        & \Delta\left(\bigvee \partial^{(\ell)}_{j i}(Y)\right)
        = \sum_{S \in \partial^{(\ell)}_{j i}(Y)}\ \Delta_{S}
        = 0, \quad \forall\, Y \in A(\alpha^{(\ell)}_{i j} \pmb{\lambda'}),
        \label{eq:relations-delta2} \\
        &  
        \nabla\left(\bigwedge \partial^{(\ell)}_{j i}(X)\right)
        = \sum_{T \in \partial^{(\ell)}_{j i}(X)} \nabla_{T} 
         = 0, \quad \forall\, X \in B(\alpha^{(\ell)}_{i j} \pmb{\lambda'}).
        \label{eq:relations-nabla2}
    \end{align}
    Moreover, we have 
    \begin{align*}
    \mathrm{Ker} \left(\Delta: \mathrm{WV}_{\pmb{\lambda'}} \bigvee V \to \mathrm{HWV}_{\pmb{\lambda}} \bigvee  V \right) &= \mathrm{span} \left\{\bigvee \partial^{(\ell)}_{j i}(Y) : \ell \in [d], i < j, Y \in A(\alpha^{(\ell)}_{i j} \pmb{\lambda'}) \right\}, \\
    \mathrm{Ker}\left(\nabla: \mathrm{WV}_{\pmb{\lambda'}} \bigwedge V \to \mathrm{HWV}_{\pmb{\lambda}} \bigwedge  V \right) &= \mathrm{span} \left\{\bigwedge \partial^{(\ell)}_{j i}(X) : \ell \in [d], i < j, X \in B(\alpha^{(\ell)}_{i j} \pmb{\lambda'}) \right\}, 
    \end{align*}
    i.e. any linear relation among $\{\Delta_T \}$ or $\{ \nabla_S \}$ follows linearly from \eqref{eq:relations-delta2} or \eqref{eq:relations-nabla2}.

    (iii) 
    Any algebraic relation of the form 
    $$\sum \alpha_T \Delta_{T^{(1)}} \vee\cdots\vee \Delta_{T^{(k)}} = 0 \qquad \text{or} \qquad \sum \beta_S \nabla_{S^{(1)}} \wedge \cdots \wedge \nabla_{S^{(k)}} = 0$$ 
    among highest weight vectors $\{\Delta \}$ and $\{\nabla \}$ (for all $m$ and $d$-tuples $\pmb{\lambda}$ of partitions of $m$) 
    follows from the linear relations listed above.
\end{theorem}

\begin{remark}\label{remark:kernel}
    As we will see below, the span of kernel comes from the action of corresponding Lie algebra, which is generated by Chevalley generators, and hence
    spanning set of the kernel can be narrowed to vectors of the form $\bigwedge \partial^{(\ell)}_{i+1, i}(X)$. The domain can also be narrowed to the (skew-)invariant subspace w.r.t. the action of corresponding parabolic subgroup $S_{\pmb{\lambda'}}$, see Remark~\ref{remark:domain}.
\end{remark}

\subsection{Raising and lowering operators}
Let $\mathfrak{g}:= \mathfrak{gl}(n)^{\times d}$ be the Lie algebra of $G:= \mathrm{GL}(n)^{\times d}$ acting on $V =(\mathbb{C}^n)^{\otimes d}$ diagonally. Then the corresponding action of Lie algebra for $v = v_1 \otimes\ldots\otimes v_m \in \bigotimes^m V$ and $A \in \mathfrak{g}$ is
$$
    A \cdot v = A v_1 \otimes\ldots\otimes v_m + \ldots + v_1 \otimes\ldots\otimes Av_m,
$$
where $Av_i$ is an ordinary matrix multiplication. The action on $\bigwedge V$ and $\bigvee V$ is analogous. 
Let $E_{i,j}$ be the $n \times n$ matrix with a single entry $1$  at $(i,j)$ and zero elsewhere. For $i < j$ the operator $E^{(1)}_{i,j}: v \mapsto (E_{i,j},0,0,\ldots) v$ is called a {\it raising operator} in direction $1$, and for $i > j$ it is called \textit{lowering operator}. Raising (or lowering) operators are consistent with the weight lattice: if $v \in A(\pmb{\lambda})$ of weight $\pmb{\lambda} = (\lambda,\ldots)$, then the vector $E^{(1)}_{i,j}(v)$ is of weight 
$\alpha^{(1)}_{ij}\pmb{\lambda} := (\lambda+e_i-e_j,\ldots)$,
i.e. increased by a positive root.
For other directions, the operators $E^{(\ell)}_{ij}, \alpha^{(\ell)}_{ij}$ are defined similarly. 
We also denote $(E^{(\ell)}_{ij})^* = E_{ji}^{(\ell)}, (\alpha^{(\ell)}_{ij})^* = \alpha_{ji}^{(\ell)}$ the pair corresponding to negative roots. 

This allows us to recover criteria of being highest weight vector. For vector $v$ of weight $\pmb{\lambda}$ we have
\begin{align}\label{eq:hwv-def1}
    v \in \mathrm{HWV}_{\pmb{\lambda}}~ \bigotimes^m V \iff \forall \ell \in [d], i < j: E^{(\ell)}_{i j}(v) = 0,
\end{align}
i.e. weight vector $v \in V$ vanishing by all raising operators is highest weight vector.

\vspace{0.5em}

The following easy fact allows us to map action of raising matrix units to $\partial$ operations.

\begin{lemma}\label{lemma:eij-action}
    Let $\pmb{\lambda}$ be a $d$-tuple of partitions of $m$. For 
    $S \in A(\pmb{\lambda})$ we have 
    $$
    \partial^{(\ell)}_{i j}(S) = \{S' \in A(\alpha^{(\ell)}_{i j} \pmb{\lambda}): \langle E^{(\ell)}_{i j} \cdot e_{S(1)}\otimes\ldots\otimes e_{S(m)}, e_{S'(1)}\otimes\ldots\otimes e_{S'(m)} \rangle = 1 \}.
    $$
    In particular, $E^{(\ell)}_{ij}\left(\bigwedge e_S\right) = \bigwedge \partial^{(\ell)}_{ij}(S)$ and $E^{(\ell)}_{ij}\left(\bigvee e_S\right) = \bigvee \partial^{(\ell)}_{ij}(S)$.
\end{lemma}

\vspace{0.5em}

\begin{lemma}[HWV relations]\label{lemma:raising1}
    Let $\pmb{\lambda}$ be a $d$-tuple of partitions of $m$ and 
    $$
        H = \sum_{X \in C^{\mathrm{lex}}(\pmb{\lambda})} h_X\  \cdot {\textstyle\bigDiamond} e_X \in \mathrm{WV}_{\pmb{\lambda}} {\textstyle\bigDiamond} V, 
    $$ 
    where we denote $(C, {\textstyle\bigDiamond})$ is $(A, \vee)$ or $(B, \wedge)$. Then
    $$
        \quad H \in \mathrm{HWV}_{\pmb{\lambda}} {\textstyle\bigDiamond} V \quad\iff\quad \forall \ell \in [d],\, i < j,\, \forall\, Y \in C^{\mathrm{lex}}(\alpha\pmb{\lambda}):\quad \left\langle H, {\textstyle\bigDiamond}\partial^*(Y) \right\rangle = 0,
    $$
    where $\partial^* = \partial^{(\ell)}_{j i}$ for $\alpha = \alpha^{(\ell)}_{i j}$. 
\end{lemma}

\begin{proof}
    Let us show the statement for the case $(C, {\textstyle\bigDiamond}) = (B, \wedge)$, the other case is signless (which is an easier) version of it. Let $(E, \partial, \alpha) := (E^{(\ell)}_{i j}, \partial^{(\ell)}_{ij}, \alpha^{(\ell)}_{ij})$ be operators for fixed $\ell \in [d], i < j$.
    We are enough to rewrite \eqref{eq:hwv-def1} into the standard basis using Lemma~\ref{lemma:eij-action}:
    \begin{align*}
        E(H) 
        = \sum_{T \in B^{\mathrm{lex}}(\pmb{\lambda'})} h_T \cdot E\left(\bigwedge e_T\right)
        = \sum_{T \in B^{\mathrm{lex}}(\pmb{\lambda'})} h_T 
            \sum_{T' \in \partial(T)} \bigwedge e_{T'}.
    \end{align*}
    Note, that $T'$ might be not a lex-table. But we still can change the order of summation without extra signs as follows: 
    $$
        E(H) = 
        \sum_{T'' \in B^{\mathrm{lex}}(\alpha\pmb{\lambda'})}\left(
            \sum_{T \in \partial^*(T'')} h_T
        \right) \cdot \bigwedge e_{T''} = 0.
    $$
    Indeed, if one derives $T \xrightarrow{\partial} T' \xrightarrow{\pi} T''$, where $\partial$ exchanges letter and $\pi$ sorts columns of $T'$, 
    then $T'' \xrightarrow{\pi^{-1}} T' \xrightarrow{\partial^*} T$ can also be derived, where the signs $\sgn(\pi) = \sgn(\pi^{-1})$ are the same. Thus, $H \in \mathrm{HWV}_{\pmb{\lambda}} \bigwedge V$ if and only if $\langle F, \wedge \partial^* T'' \rangle = 0$ for arbitrary $T'' \in B^{\mathrm{lex}}(\alpha\pmb{\lambda'})$, as desired. The case $(C,{\textstyle\bigDiamond}) = (A, \vee)$ is similar, without taking care of signs.
\end{proof}

\vspace{1em}

\begin{proof}[Proof of Theorem~\ref{th:relations1}]
    All four cases of the relations have identical proofs. Let us show it for odd $d$ and the map $\Delta: \mathrm{WV}_{\pmb{\lambda'}} \bigwedge V \to \mathrm{HWV}_{\pmb{\lambda}} \bigvee V$. 

    Fix $\ell \in [d]$ and $1 \le i < j \le \ell((\lambda^{(\ell)})')$ and set $(\partial, \partial^*, \alpha):=(\partial^{(\ell)}_{ij}, \partial^{(\ell)}_{ji}, \alpha^{(\ell)}_{ij})$.
    We apply Lemma~\ref{lemma:raising1} to the spanning set $\{\nabla_S\}$ of 
    $\mathrm{HWV}_{\pmb{\lambda'}} \bigwedge V$ 
    which is the dual space of $\mathrm{HWV}_{\pmb{\lambda}} \bigvee V$. 
    Since for all $S \in A(\pmb{\lambda})$ we can write $\nabla_S = \sum_{T \in B^{\mathrm{lex}}(\pmb{\lambda'})} b(S,T) \cdot \wedge e_T$, it follows that using the duality of highest weight vectors $\Delta$ and $\nabla$ we have 
    \begin{align*}
        \forall X &\in B(\alpha \pmb{\lambda'}),\ \forall S \in A(\pmb{\lambda}): &  &\sum_{T \in \partial^*(X)} b(S,T) = 0 &\iff\\
        \forall X &\in B(\alpha \pmb{\lambda'}),\ \forall S \in A(\pmb{\lambda}): & &\sum_{T \in \partial^*(X)} a(T,S) = 0 &\iff\\
        \forall X &\in B(\alpha \pmb{\lambda'}): & &\sum_{T \in \partial^*(X)} \Delta_{T} = \Delta\left(\bigwedge{\partial^*(X)}\right) = 0,
    \end{align*}
    as desired.

    Let us now prove the statement about the kernel of the map $\Delta$. (Again, the other cases have identical proofs.) 
    Denote $B_{\pmb{\lambda'}} = \mathrm{WV}_{\pmb{\lambda'}} \bigwedge V$, $HB_{\pmb{\lambda'}} = \mathrm{HWV}_{\pmb{\lambda'}} \bigwedge V$
    and $HB^{\perp}_{\pmb{\lambda'}} \subseteq B_{\pmb{\lambda'}}$ the subspace spanned by vectors $\bigwedge {\partial^*(X)}$. We will show that $\mathrm{Ker}(\Delta) = HB^{\perp}_{\pmb{\lambda'}}$.
    
    By Lemma~\ref{lemma:raising1} we know that $HB^{\perp}_{\pmb{\lambda'}} \oplus HB_{\pmb{\lambda'}} = B_{\pmb{\lambda'}}$, where $HB_{\pmb{\lambda'}} = \mathrm{HWV}_{\pmb{\lambda'}} \bigwedge V$.
    As we already showed above we have $\Delta\left(\bigwedge {\partial^*(X)}\right) = 0$, and hence $HB^{\perp}_{\pmb{\lambda'}} \subseteq \ker(\Delta)$. By Theorem~\ref{th:isomorphism}, we have $\ker(\Delta)\cap HB_{\pmb{\lambda'}} = 0$, and we conclude that $\ker(\Delta) \subseteq HB^{\perp}_{\pmb{\lambda'}}$ which now implies that $\mathrm{Ker}(\Delta) = HB^{\perp}_{\pmb{\lambda'}}$. 
    
    For the statement (iii), 
to describe any (homogeneous) relation of the form
$$
    \sum \alpha_T \Delta_{T^{(1)}} \vee \cdots \vee \Delta_{T^{(k)}} = 0 \qquad \text{or}\qquad \sum \beta_S \nabla_{S^{(1)}} \wedge \cdots \wedge \nabla_{S^{(k)}} = 0,
$$
we can rewrite it in the linear form due to ring product rule from Proposition~\ref{prop:ring}:
$$
    \sum \alpha_T \Delta_{(T^{(1)} \oplus \cdots \oplus T^{(k)})} = 0, 
    \qquad \text{or}\qquad 
    \sum \beta_S \nabla_{(S^{(1)}\oplus \cdots \oplus S^{(k)})} = 0.
$$
Thus all algebraic relations follow from linear relations among the highest weight vectors $\Delta$ and $\nabla$.     
\end{proof}

\section{Power expansions of the Cayley form for odd $d$}\label{sec:odd-d}
Let $d$ be odd, $\pmb{\lambda}$ be $d$-tuple of partitions of $m$. In this section we consider the dual spaces
$$
    \mathrm{HWV}_{\pmb{\lambda'}}\bigwedge^m (\mathbb{C}^k)^{\otimes d}\qquad\text{and}\qquad \mathrm{HWV}_{\pmb{\lambda}}\bigvee^m (\mathbb{C}^n)^{\otimes d}
$$
where we replaced $n \to k$ for wedge space to distinguish the spaces. This imposes restriction $\pmb{\lambda} \subseteq (n \times k)^{\times d}$.

\subsection{$\mathrm{SL}$-invariants} 
If we specialize to the rectangular weights $\pmb{\lambda} \to (n \times k)^d$, the highest weight polynomials (w.r.t. $\mathrm{GL}(n)^{\times d}$) and forms (w.r.t $\mathrm{GL}(k)^{\times d}$) become invariants w.r.t. $\mathrm{SL}(n)^{\times d}$ and $\mathrm{SL}(k)^{\times d}$, respectively. With these weights we will use the following notation for hypermatrices:
\begin{align*}
    \mathsf{A}_d(n,k) &:= A^{\lex}((n \times k)^d) = \mathsf{T}_{\mathbb{N}}((n \times k)^d),\\
    \mathsf{B}_d(k,n) &:= B^{\lex}((k \times n)^d) = \mathsf{T}_{01}((k \times n)^d),
\end{align*}
i.e. the first argument corresponds to the length of the cube and the second corresponds to the sum of entries in each slice.
These set equalities are identified via $d$-line notation. Note that the $(0,1)$-hypermatrix $T \in \mathsf{B}_d(n,k)$ is determined by its support $\mathrm{supp}(T) \subseteq [k]^d$.

Denote the special class of Kronecker coefficients at rectangular shapes:
$$
    g_d(n, k) := g(\underbrace{n\times k,\ldots, n \times k}_{d\text{ times}})
$$
and recall that $g_d(n,k) = \dim \mathrm{HWV}_{(n \times k)^d}\bigvee^{nk} (\mathbb{C}^n)^{\otimes d} = \dim \mathrm{HWV}_{(k \times n)^d}\bigwedge^{nk} (\mathbb{C}^k)^{\otimes d}$. 
The ring of $\mathrm{SL}$-invariant polynomials over $(\mathbb{C}^n)^{\otimes d}$ has the grade decomposition
$$
    \mathbb{C}[(\mathbb{C}^n)^{\otimes d}]^{\mathrm{SL}(n)^d}
    \cong
    \bigoplus_{k \ge 0} \mathrm{HWV}_{(n \times k)^d} \bigvee^{nk} (\mathbb{C}^n)^{\otimes d},
$$
but we fix $k$ (multiple of a degree) and iterate over $n$ instead. The resulting sequence of rectangular Kronecker coefficients $\{g_d(n,k)\}_{n \ge 0}$ counts grades' dimension of the following subspace of $\bigwedge (\mathbb{C}^k)^{\otimes d}$:
$$
    \bigoplus_{n \ge 0} \mathrm{HWV}_{(k \times n)^d} \bigwedge^{nk} (\mathbb{C}^k)^{\otimes d}.
$$
\begin{remark}
    We studied the sequence $g_d(n,k)$ in \cite{ayfund}. In particular, we showed the symmetry $g_d(n,k) = g_d(k^{d-1}-n, k)$ for $n \in [0,k^{d-1}]$ (note that $g_d(n,k) = 0$ for $n > k^{d-1}$) and conjectured that the sequence is unimodal for even $k$. In \cite{ayuni} we refined this unimodality conjecture for arbitrary weights $\pmb{\lambda}$ and proved refined version of this conjecture for $k = 2$.
\end{remark}

It is known that $g_d(1, k) = g_d(k^{d-1}, k) = 1$, which suggests the existence of unique invariants in the corresponding graded pieces of polynomials and forms. These invariants are particularly important. Let us discuss them.

{The first identity} $g_d(1,k) = 1$ for polynomials is trivial: there is a unique homogenious polynomial of degree $k$ in $\mathbb{C}[(\mathbb{C}^1)^{\otimes d}]$. While for forms, there is a special highest weight vector $\omega \in \mathrm{HWV}_{(k \times 1)^d}\bigwedge (\mathbb{C}^k)^{\otimes d}$, which we call the \textit{Cayley form} given by:
$$
    \omega = \frac{1}{k!}\sum_{\pi_1,\ldots\pi_d \in S_k} \sgn(\pi_1\cdots\pi_d) \bigwedge_{i=1}^k e_{\pi_1(i),\ldots,\pi_d(i)} = \nabla_{\begin{pmatrix}
        1^k\\
        \vdots\\
        1^k
    \end{pmatrix}} \in \bigwedge^k (\mathbb{C}^k)^{\otimes d},
$$
which is (up to a scale) unique highest weight vector of weight $(k \times 1)^d$. 
Note that the product in wedge space corresponds to the direct sum operation $\bigoplus$ on tables by Proposition~\ref{prop:ring}, and therefore, for each $n \le k^{d-1}$ we have 
\begin{align*}
    \omega^{n} = \nabla_{\begin{pmatrix}
        1^k \ldots n^k\\
        \vdots\\
        1^k \ldots n^k
    \end{pmatrix}} =: \nabla_{I_{n,k}} \in \mathrm{HWV}_{(n\times k)^d}\bigwedge^{nk} (\mathbb{C}^k)^{\otimes d}
\end{align*}
where we denote the corresponding $d \times nk$ table by $I_{n,k} \in \mathsf{A}_d(n,k)$. 

\vspace{0.5em}

{For the second identity} $g_d(k^{d-1}, k)=1$, there is a unique invariant polynomial $\Delta_{F_{d,k}}$, where $F_{d,k} \in \mathsf{B}_d(k, k^{d-1})$ is the $d \times k^{d}$ table with all distinct columns. In particular, $\mathrm{supp}(F^{(d)}_k) = [k]^d$. For example, for $d = 3$ we have the following table:
$$
    F_{3,k} = \begin{pmatrix}
        ~ & 1^{k^{2}} & ~ &\ldots & ~ & k^{k^{2}} & ~
        \\
        1^{k} & \ldots & k^{k} &
        \ldots & 
        1^{k} & \ldots & k^{k}
        \\
        12\ldots k & \ldots & 12\ldots k 
        & \ldots & 
        12\ldots k & \ldots & 12\ldots k 
    \end{pmatrix}.
$$
\begin{remark}
This invariant was introduced by B\"urgisser and Ikenmeyer in \cite{bi}.
This polynomial is a \textit{fundamental} invariant, in the sense that it is the smallest degree invariant in the ring $\mathbb{C}[(\mathbb{C}^{k^{d-1}})^{\otimes d}]^{\mathrm{SL}(n)^d}$ of degree $k^d$. \end{remark}

Note that the unique invariant in $\bigwedge^{k^d} (\mathbb{C}^{k})^{\otimes d}$ is the volume form $\bigwedge e_{[k]^d}$.

\subsection{Expansion of the Cayley form}
Using the main duality from Section~\ref{sec:duality} we give a conceptual explanation to power expansions of the Cayley form $\omega$. 

\begin{theorem}\label{thm:omega}
    Let $d \ge 3$ be odd. For each $n \le k^{d-1}$ we have:
    $$
        \omega^n = (-1)^{n \times k} \sum_{T\in \mathsf{B}_d(n,k)} \Delta_T(\mathsf{I}_n) \bigwedge e_T
    $$
    where sum runs over all $d \times nk$ tables of weight $(n \times k)^d$. In particular, for $n = k^{d-1}$ we have:
    $$
        \omega^{k^{d-1}} = (-1)^{k^{d-1} \times k} \Delta_{F_{d,k}}
        (\mathsf{I}_{k^{d-1}})
        \bigwedge e_{[k]^d}. 
    $$
\end{theorem}
\begin{proof}
    Since $\omega^n = \nabla_{I_{n,k}}$ as noted above, we have
    $$\nabla_{I_{n,k}} = \sum_{T \in \mathsf{B}_d(k,n)} b(T, I_{n,k}) \bigwedge e_T.$$
    By Theorem~\ref{th:duality2} we have $b(I_{n,k}, T) = (-1)^{n \times k} a(T, I_{n,k})$. 
    Since evaluation of the polynomial $\Delta_T$ at the unit tensor $\mathsf{I}_n$ is equal to the coefficient $a(T, I_{n,k})$, the result follows.
\end{proof}

\begin{remark}
We first proved this result about power expansions of $\omega$ in \cite{ayfund} by a different rather long combinatorial argument analyzing signs of Latin hypercubes. Now we have a conceptual proof of this result. Note also that this result is important for showing positivity of Kronecker coefficients. 
\end{remark}

\subsection{Latin hypercubes}\label{sec:latin}
    Evaluations of invariants at unit tensors has a nice combinatorial interpretation. For $T \in \mathsf{B}_d(n,k)$, let $L$ be a hypermatrix with $\mathrm{supp}(L) = T \subseteq [k]^d$. 
    
    We call $L$ a \textit{partial Latin hypercube}, if non-zero entries of any fixed slice form a permutation, conventionally collected in lexicographical order of $[k]^d$. When $\mathrm{supp}(T) = [k]^d$, we call $L$ a \textit{Latin hypercube}. Let $\mathsf{L}_d(T)$ and $\mathsf{L}_d(k)$ be the sets of partial Latin hypercubes with support $T$ and $[k]^d$, respectively. 
    
    The \textit{sign} or \textit{signature} $\sgn(L)$ of a partial Latin hypercube $L$ with support $T$ is defined to be the product over all slices (in all directions) of the slice permutation signs. We then define the sums
    $$
        \mathrm{AT}_d(T) := \sum_{L \in \mathsf{L}_d(T)} \sgn(L), \qquad 
        \mathrm{AT}_d(k) := \sum_{\mathsf{L}_d(k)} \sgn(L),
    $$
    which we call {$d$-dimensional Alon--Tarsi numbers}. 
    The following proposition is straightforward from Proposition~\ref{prop:interpretation}.
    \begin{proposition}
    We have 
        $$\Delta_T(\mathsf{I}_n) = a(T, I_{n,k}) = (-1)^T \mathrm{AT}_d(T)$$ for any $d$ and $T \in \mathsf{B}_d(k, n)$. Moreover, $\mathrm{AT}_d(T) = 0$ for odd $k$. 
    \end{proposition}
    
 \begin{remark}
 We refer to \cite{ayfund} for more on Latin hypercubes. 
 \end{remark}

    \begin{remark}
        While it is evident that $\Delta_{F_{d,k}} \neq 0$ and $\Delta_{F_{d,k}}(\mathsf{I}_n) = \pm \mathrm{AT}_{d}(k)$, there are other nonzero coefficients in the fundamental invariant $\Delta_{F_{d,k}}$. Similarly, $\nabla_{I_{k^{d-1}, k}} = \mathrm{AT}_d(k) \cdot \bigwedge e_{[k]^d}$ and there are other indices $S \in \mathsf{A}_d(k^{d-1}, k)$ for which $\nabla_{S} \neq 0$. For such $S$, in fact, $\nabla_S \neq 0$ if and only if $\langle \Delta_{F_{d,k}}, \vee e_S \rangle \neq 0$.    
    \end{remark}
    
\begin{remark}[On the Alon--Tarsi conjecture and Kronecker positivity]    
    Recall that the Alon--Tarsi conjecture on Latin squares \cite{at} (cf. \cite{hr}) 
    states that $\mathrm{AT}_2(k) \neq 0$ for even $k$. For $d = 3$, a similar problem about Latin cubes was posed by Burgisser--Ikenmeyer \cite{bi} asking if $\mathrm{AT}_3(k) \neq 0$ for even $k$ (which holds for $k = 2,4$). As we studied in \cite{ayfund}, this problem can be generalized for any $d$ given that $\mathrm{AT}_d(2) \ne 0$  and that the statement $\mathrm{AT}_d(k) \ne 0$ for even $k$ and odd $d$ implies positivity of Kronecker coefficients; namely, using the power expansions of the Cayley form $\omega$ we showed that $\mathrm{AT}_2(k) \ne 0$ implies that $g_d(n,k) > 0$ for all $n \le k$, while $\mathrm{AT}_3(k) \ne 0$ implies that $g_d(n,k) > 0$ for all $n \le k^{d-1}$; furthermore, unconditionally for even $k$ and odd $d$ we have $g_d(n,k) > 0$ for $n \le \max\{2^{d-1}, \sqrt{n}/2 - 1 \}$ \cite{ayfund, ayuni}, which uses the facts that $\mathrm{AT}_2(p \pm 1) \ne 0$ for odd primes $p$ \cite{dri1, glynn} and $\mathrm{AT}_d(2) \ne 0$. 
\end{remark}



\section{Power expansions of Cayley's first hyperdeterminant for even $d$}\label{sec:even-d}
    In this section we let $d$ to be even and consider spaces:
    $$
        \mathrm{HWV}_{\pmb{\lambda'}} \bigvee (\mathbb{C}^n)^{\otimes d} 
        \qquad \text{and} \qquad 
        \mathrm{HWV}_{\pmb{\lambda}} \bigvee (\mathbb{C}^k)^{\otimes d}
    $$
    Recall that the highest weight and weight spaces can be described as follows:
    The symmetry of Kronecker coefficients $g(\pmb{\lambda}) = g(\pmb{\lambda'})$ for even $d$ confirms the connection between underlying highest weight spaces which we established earlier: 
    For $(S,T) \in A(\pmb{\lambda}) \times B(\pmb{\lambda})$ we have
    \begin{align*}
        \Delta_{S} = \sum_{S' \in A(\pmb{\lambda'})} 
            \langle \Delta_{S'}, \bigvee e_S \rangle \bigvee e_{S'},\qquad
        \nabla_{T} = \sum_{T' \in B(\pmb{\lambda'})} 
            \langle \nabla_{T'}, \bigwedge e_T \rangle \bigwedge e_{T'},
    \end{align*}
    where $S'$ and $T'$ run over $d$-tables of conjugate weight.

\subsection{Expansion of Cayley's first hyperdeterminant}
We are going to show the next application of the duality on invariant polynomials. We will use the tables  $I_{n,k}$ and $F_{d,k}$ defined in the previous section.

Let us recall Cayley's first hyperdeterminant $\delta = \delta_{d,k}$ defined as follows: 
$$
    \delta := \frac{1}{k!}\sum_{\sigma \in (S_k)^d} \sgn(\sigma) \bigvee_{i=1}^k e_{\sigma(i)}
    = \Delta_{I_{1,k}} \in \mathbb{C}[(\mathbb{C}^k)^{\otimes d}]_{k},
$$
where $\sgn(\sigma) = \sgn(\sigma_1)\cdots\sgn(\sigma_d)$ for $\sigma = (\sigma_1, \ldots, \sigma_d)$. Note that this formula becomes identical zero for odd $d$. More recognized form would be:
$$
    \delta^*(X) = \frac{1}{k!}\sum_{\sigma_1,\ldots,\sigma_d \in S_k} \sgn(\sigma_1\cdots\sigma_d) \prod_{i=1}^k X_{\sigma_1(i)\ldots\sigma_d(i)}
$$

\begin{remark}\label{remark:hdet}
    The invariant $\omega$ (discussed in previous section) can be regarded as a $k$-form version of Cayley's first hyperdeterminant $\delta$: for tensor $X = \sum_{i=1}^n e_i \otimes X_i \in (\mathbb{C}^n)^{\otimes d}$:
$$
    \delta^*(X) = \omega^*(X_1, \ldots, X_k)
$$
for 
even $d$. This invariant was introduced by Cayley \cite{cay,cay2}. 
Notably, it is the smallest degree $\mathrm{SL}(k)^{\times d}$-invariant polynomial over $(\mathbb{C}^k)^{\otimes d}$ for even $d$ and is unique up to a scalar. For odd $d$, the above formula gives identical zero; moreover, the smallest degree of an invariant is always at least $n\lceil n^{1/(d-1)} \rceil$ \cite{ayfund} which corresponds to the degree of the fundamental invariant $\Delta_{F_{d,k}}$ for $n = k^{d-1}$. 
\end{remark}

\begin{remark}
The hyperdeterminant $\delta$ can also be characterized as the unique function (up to a scalar) satisfying multi-linearity and skew-symmetry in argument's slices in any fixed direction, see \cite{ayhdet}.
\end{remark} 

In particular, the powers of $\delta$ can be written as
$
    \delta^n = \Delta_{I_{n,k}},
$
since the product in the polynomial ring is translated to $\oplus$ operation on tables by Proposition~\ref{prop:ring}. It is known that for $d = 2$, the coefficient of the power of the ordinary determinant $\det^k$ at $e_{11}e_{12}\ldots e_{kk}$ is equal to the Alon--Tarsi number $(-1)^{\floor{k/2}}\mathrm{AT}_2(k)$. 
 We prove the following generalization of this fact, which becomes an analogue Theorem~\ref{thm:omega} about the power expansions of $\omega$ for even $d$. 

\begin{theorem}
    Let $d \ge 2$ be even. For any $n \ge 1$ we have:
    $$
        \delta^n = \sum_{S \in \mathsf{A}_{d}(k, n)} \Delta_{S}(\mathsf{I}_n) \cdot \bigvee e_S,
    $$
    In particular, the coefficient of $\delta^{k^{d-1}}$ at $\bigvee e_{[k]^d}$ is equal to $\pm \mathrm{AT}_d(k)$.
\end{theorem}
\begin{proof}
    Let us rewrite the power of the hyperdeterminant in the form of $\Delta$-polynomial: $\delta^{n} = \Delta_{I_{n, k}}$ for $I_{n,k} \in \mathsf{A}_d(n,k)$. For each $S \in \mathsf{A}_d(k, n)$, the coefficient at $\bigvee e_S$ in $\Delta_{I_{n,k}}$ is the number $a(I_{n, k}, S)$, which by the symmetry of coefficients from Lemma~\ref{lemma:ats}(d) satisfies
    $$a(I_{n,k}, S) = (-1)^{\pmb{\lambda}}a(S, I_{n,k}) = (-1)^{\pmb{\lambda}}\Delta_{S}(\mathsf{I}_n),$$ 
    where $\pmb{\lambda} = (n\times k)^d$. Note, that $(-1)^{\pmb{\lambda}} = ((-1)^{n \times k})^d = 1$, since $d$ is even.
    In particular, for $n = k^{d-1}$ and $F = F_{d,k}$ we have
    $$
    a(I_{k^{d-1}, k}, F) = a(F, I_{k^{d-1}, k}) = \Delta_{F}(\mathsf{I}_n) = (-1)^{F} \mathrm{AT}_d(k),
    $$
    and the proof follows.
\end{proof}

Power expansions $\delta^{n}$ contain other curious coefficients.
It turns out that the appropriate power of $\delta^n$ contains both the number of $c$-dimensional Latin hypercubes and the signed sum of Latin hypercubes $\mathrm{AT}_c(k)$ for any even $c \le d/2$. 

\begin{proposition}
    Let $d$ and $ c \in [2, d/2]$ be both even. Consider the following $d \times k^{c}$ tables:
    $$
    T_{odd} = \begin{pmatrix}
        ~ & ~ & ~\\
        ~ & F_{c,k} & ~\\
        ~ & ~ & ~\\
        \hline
        12 \cdots k & \cdots & 12\cdots k\\
        \vdots & \vdots & \vdots\\
        \vdots & \vdots & \vdots\\
        \vdots & \vdots & \vdots\\
        12\cdots k & \cdots & 12\cdots k
    \end{pmatrix}, 
    \qquad
    T_{even} = \begin{pmatrix}
        ~ & ~ & ~\\
        ~ & F_{c,k} & ~\\
        ~ & ~ & ~\\
        \hline
        ~ & ~ & ~\\
        ~ & F_{c,k} & ~\\
        ~ & ~ & ~\\
        \hline
        12\cdots k & \cdots & 12\cdots k\\
        \vdots & \vdots & \vdots\\
        12\cdots k & \cdots & 12\cdots k
    \end{pmatrix}.
    $$ 
    Then the coefficients of $\delta^{k^{c-1}}$ at $\bigvee e_{T_{even}}$ is $|\mathsf{L}_c(k)|$ and at $\bigvee e_{T_{odd}}$ is $\mathrm{AT}_c(k)$.
    In fact, the same holds for $T_{odd}$ if $c > d/2$.
\end{proposition}
\begin{proof}
    Let $I = I_{k^{c-1},k}$. The desired coefficients are then $a(I, T_{even})$ and $a(I, T_{odd})$. Let us look into their expansions by definition \eqref{eq:ats-bst}. Note that 
    $$\sgn_{T_{odd}}(\pi I) = \sgn_{F_{c,k}}(\pi I) \text{  and  }\sgn_{T_{even}}(\pi I) = |\sgn_{F_{c,k}}(\pi I)|$$ 
    for any $\pi \in S_{k^{c}}$, since each row is repeated \textit{parity} (i.e. odd or even) number of times in the table $T_{parity}$. Thus, 
    $$a(I, T_{odd}) = a(T_{odd}, I) = a(F_{c,k}, I') = \pm\mathrm{AT}_d(k),$$ 
    where $I'$ is formed by the first $c$ rows of $I$. Similarly, for $a(T_{even}, I)$ each non-vanishing term in the sum has a positive sign, giving $|\mathsf{L}_c(k)|$.
\end{proof}

\begin{corollary}
    For the $4 \times k^2$ tables
    $$
    T_{even} = \begin{pmatrix}
        11\cdots 1 \cdots kk \cdots k\\
        11\cdots 1 \cdots kk \cdots k\\
        12\cdots k \cdots 12\cdots k\\
        12\cdots k \cdots 12\cdots k
    \end{pmatrix}, \qquad
    T_{odd} = \begin{pmatrix}
        11\cdots 1 \cdots kk \cdots k\\
        11\cdots 1 \cdots kk \cdots k\\
        11\cdots 1 \cdots kk \cdots k\\
        12\cdots k \cdots 12\cdots k
    \end{pmatrix}
    $$
    the coefficient of $\delta_{4,k}^k$ at $\bigvee e_{T_{even}}$ is equal to the number of Latin squares $|\mathsf{L}_2(n)|$ and at $\bigvee e_{T_{odd}}$ is equal to $(-1)^{k/2}\mathrm{AT}_2(k)$.
\end{corollary}



\end{document}